\documentclass[11pt,final]{amsart}
\usepackage{amssymb} 
\usepackage[colorlinks=true, citecolor=blue, linkcolor=blue, urlcolor=blue]{hyperref}
\usepackage{comment, cite}

\usepackage{ifdraft}
\ifoptionfinal{
\usepackage[disable]{todonotes}
}{
\usepackage[norefs, nocites]{refcheck}

\usepackage[notref, notcite]{showkeys}
\usepackage[bordercolor=white, color=white]{todonotes}
\usepackage{bm}
}

\makeatletter\providecommand\@dotsep{5}\def\listtodoname{List of Todos}\def\listoftodos{\hypersetup{linkcolor=black}\@starttoc{tdo}\listtodoname\hypersetup{linkcolor=blue}}\makeatother
\usepackage{mathtools}
\mathtoolsset{showonlyrefs}
\newtheorem{lemma}{Lemma}[section]
\newtheorem{proposition}{Proposition}[section]
\newtheorem{theorem}{Theorem}[section]

\newtheorem{definition}{Definition}[section]
\theoremstyle{remark} 

\newtheorem{remark}{Remark}[section]

\def\R{\mathbb R}

\renewcommand{\i}{\mathrm{i}}
\newcommand{\Id}{\textrm{\rm Id}}
\renewcommand{\r}[1]{{\eqref{#1}}}

\newcommand{\be}[1]{\begin{equation}\label{#1}}
\newcommand{\ee}{\end{equation}}
\renewcommand{\d}{d}

\renewcommand{\mathbf}{\boldsymbol}

\DeclareMathOperator{\supp}{supp}

\date{\today}
\title[The backscattering problem for time dependent potentials]{The backscattering problem for time-dependent potentials}

\author[M. Nursultanov]{Medet Nursultanov}
\address {Department of Mathematics and Statistics, University of Helsinki, Helsinki, and Institute of Mathematics and Mathematical Modeling, Almaty, Kazakhstan}

\author[L. Oksanen]{Lauri Oksanen}
\address{Department of Mathematics, University of Helsinki, 
Helsinki, Finland}
\author[P. Stefanov]{Plamen Stefanov}
\address{Department of Mathematics, Purdue University, West Lafayette, IN 47907, USA}

\begin{document}
\begin{abstract}
	We study the inverse backscattering problem for time-dependent potentials. We prove uniqueness and Lipshitz stability for the recovery of small potentials.
\end{abstract}
\maketitle



\section{Introduction} 
Let   $q(t,x)$ be smooth and supported in the cylinder $\R\times\Omega$, where $\Omega\subset B(0,\rho) := \{|x|<\rho\}$ with some $\rho>0$ is a fixed domain un $\R^n$.  We study the inverse back-scattering problem for the wave equation
\be{eq1}
(\partial_t^2-\Delta+q(t,x))u=0,   \quad (t,x)\in \R\times \R^n,
\ee
$n\ge3$, odd. We show that small enough potentials $q$ are \textit{stably} recoverable from the data. 

Results for stationary potentials $q(x)$ have been proven in \cite{Eskin-Ralston_2D, Eskin-Ralston_3D,  S-CPDE, RuizV-05, Wang-02}. Even though stability (say, of conditional H\"older type) has not been stated explicitly there (see also \cite{SU-JFA09} for a related result), it follows from the fact that the linearization of the problem near $q=0$ is essentially the Fourier transform of $q$, see, e.g., \cite{S-CPDE}. In terms of uniqueness, the best known result is generic uniqueness so far. 

The inverse problem of recovery of $q(t,x)$ from ``near-field'' scattering data, closely related to the inverse scattering one but not restricted to back-scattering, has been studied in \cite{MR1004174, Ramm-Sj, Salazar_13, waters2014stable, Aicha_15, Kian2016}, and other works.  Uniqueness is known, for example for potentials supported in a cylinder as above, with a tempered growth in $t$, as shown in  \cite{MR1004174}.  One of the  techniques is to extract the light-ray transform from the data, which relies on forward scattering, and invert it, see, e.g., \cite{St-Yang-DN} for an even more general situation.  That transform does not see timelike singularities however, see \cite{S-support2014, LOSU-Light_Ray, SU-book} which makes it unstable. In view of that, the possibility of a \textit{stable} recovery of $q$ remained unclear. In \cite{Rakesh-time-dep-21}, it was shown that a similar boundary value problem, with inputs plane waves as below, and the output measured at a fixed time $t=T$ in the whole $\R^n_x$, provides Lipschitz stable recovery. The proof is based on Carleman estimates. 

Even though  forward propagating rays do not see all singularities,  broken rays reflecting from the interior could, at least on the principal level. Back-scattering provides such a geometry, in particular. The main reason why one can expect a stable recovery in this case is the following. Plane waves can only possibly detect singularities conormal to them, which are lightlike, indeed. On the other hand, a linearization of the backscattering data near $q=0$ is an integral over the product of one such incoming and one outgoing wave.  That product, on the principal level, is supported on the intersection of such two hyperplanes in timespace, which is a delta on a codimension two (vs. one) hyperplane, see Figure~\ref{time-dep-pot_pics1}, where it  looks like a line. That hyperplane has a richer subspace of conormals and can possibly detect non-necessarily lightlike singularities. Varying the incident direction of the incoming wave provides a complete set of conormals. We refer to the discussion in section~\ref{sec_idea} as well. 

We are restricted to small potentials, and as we pointed out already, even for stationary $q(x)$, the uniqueness without that assumption is a well-known open problem. It seems feasible that our methods could help  prove local generic uniqueness (and stability) in line with the stationary results in \cite{Eskin-Ralston_3D,Eskin-Ralston_2D, S-CPDE}.

\section{Main Results}
We describe the scattering amplitude for \r{eq1} briefly in order to formulate the main theorem. In Appendix~\ref{sec_appendix}, we review the scattering theory for \r{eq1} in more detail. 

We are sending waves 
$\delta(t+s-x\cdot\omega)$, where $s$ is a delay parameter, $|\omega|=1$,  and $t\ll 0$; let them propagate and scatter, and measure them at infinity at directions $\omega'$ and delay time $s'$. The \textit{scattering amplitude} $A^\sharp(s',\omega',s,\omega)$, see Definition~\ref{def_sc_a} and Proposition~\ref{pr_A}, measures the difference between the wave we sent and the scattered one. Taking $\omega' =-\omega$, we measure the response in the direction opposite of the incoming one. 
If we have two potentials, $q_1$ and $q_2$, we denote the corresponding quantities by the subscripts $1$ and $2$.
 
To state our main results, we introduce the change of variables
\begin{equation}\label{change_var}
        \sigma = \frac{s - s'}{2},
        \qquad
        \sigma' = \frac{s + s'}{2}.
\end{equation}
By $\tilde{A}_1^\sharp(\sigma',\sigma,\omega)$ and $\tilde{A}_2^\sharp(\sigma',\sigma,\omega)$, we denote the functions $A_1^\sharp(s',-\omega,s,\omega)$ and $A_2^\sharp(s',-\omega,s,\omega)$ in the new variables. Our main result is the following. 
\begin{theorem}\label{thm_1}
There exists $\varepsilon>0$ and $k>0$ such that if
\begin{equation*}
    \|q_1\|_{C^k(\mathbb{R}\times\bar\Omega)} <\varepsilon,
    \qquad
    \|q_2\|_{C^k(\mathbb{R}\times\bar\Omega)}<\varepsilon,
\end{equation*}
then the identity $\tilde{A}_1^\sharp = \tilde{A}_2^\sharp$ implies $q_1=q_2$.
Moreover, under the same assumptions on $q_1$, $q_2$,  there exists a constant $C_\Omega>0$ such that
\[
\|q_1-q_2\|_{L^\infty(\R;\; L^2(\R^n))}\le C_\Omega \|\tilde{A}_1^\sharp-\tilde{A}_2^\sharp\|_{L^\infty(\mathbb{R}_{\sigma'}; L^2(\mathbb{S}_\omega^{n-1}; H^{(n-1)/2}(\mathbb{R}_\sigma)))}.
\]
\end{theorem}

Note that we could use other norms above using complex interpolation under the assumptions of the theorem but the price for that is to make the estimate of conditional H\"older type, i.e., to have $\|A_1^\sharp-A_2^\sharp\|^\mu$ above with some $\mu\in (0,1)$. 


\section{Proofs}\label{sec_3}
\subsection{A pseudo-linearization identity} 
We review the scattering theory for time-dependent potentials in Appendix~\ref{sec_appendix} mostly following \cite{CooperS,CooperS84,MR1004174} with some additions as well. We sketch the main notions below. 

We send a plane wave $\delta(t+s-x\cdot\omega)$ to the perturbation, and let it interact with the potential. More precisely, we are solving
\be{TD_14}
(\partial_t^2-\Delta+q(t,x))u^-=0,   \quad u^-|_{t<-s-\rho }= \delta(t+s-x\cdot\omega).
\ee
Then we set 
\be{TD_14a}
u_\textrm{sc}^-= u^--\delta(t+s-x\cdot\omega). 
\ee
The distribution $u_\textrm{sc}^-$ (which is actually a function, see Proposition~\ref{pr2}) would be automatically outgoing by   Definition~\ref{def_out}, since it vanishes for $t\ll0$. Then we could compute the asymptotic wave profile $u^{-,\sharp}_\textrm{sc} (s',\omega';s,\omega)$ of $u_\textrm{sc}^- (t,x;s,\omega)$,   which would give us the analog of the scattering amplitude, see section~\ref{sec_OS}. As in the stationary case, we expect this to be ``essentially'' the kernel of the scattering operator minus identity. This is true, indeed, at least when the scattering operator exists as a bounded one as we show in Theorem~\ref{TD_thm_sc1}. One defines the scattering amplitude $A^\sharp(s',\omega';s,\omega)$ by canceling some constant and ignoring some   $s'$ derivatives. 

We need to define the time-reversed  analog of $u^-$ above, which we will denote by $u^\textrm{+} (t ,x;s,\omega)$. It solves
\be{TD_14in}
(\partial_t^2-\Delta+q(t,x))u^\textrm{+} =0,   \quad u^\textrm{+} |_{t>-s+\rho }= \delta(t+s-x\cdot\omega).
\ee
We want to warn the reader about a possible confusion caused by the terms incoming/outgoing. The solution $u$ of \r{TD_14}, which we denote by $u^-$ below, is the response to an incident plane wave and it is neither incoming nor outgoing by  Definition~\ref{def_out}. On the other hand, $u_\textrm{sc}^-=u^-_\textrm{sc} $ is outgoing. Similarly, $u^+$ is neither but $u^+_\textrm{sc}$, defined as in \r{TD_14a} but with $u$ replaced by $u^+$, is incoming.

Let $q_1$ and $q_2$ be two such potentials, and denote the corresponding quantities with subscripts $1$ and $2$. We have the following formula, proven also in  \cite{MR1082237} for $n=3$,  generalizing that in \cite{S-CPDE}, where the potentials are time-independent. 
\begin{proposition}\label{pr1}
We have 
\be{pr1-est}
(A_1^\sharp - A_2^\sharp )(s',\omega';s,\omega) = \int (q_1-q_2) (t ,x )u_1^-(t ,x,s,\omega)u_2^+(t,x, s',\omega' )\, \d t\,\d x,
\ee
where $u_1^-$ solves \r{TD_14} with $q=q_1$, and  $u_2^+$ solves \r{TD_14in} with $q=q_2$.
\end{proposition}

\begin{proof}
Start with
\be{pr1-2}
U_1(t,s)-U_2(t,s) = \int_s^t U_1(t,\sigma)(Q_1(\sigma)-Q_2(\sigma))U_2(\sigma,s)\,\d \sigma,
\ee
which can be obtained by applying the Fundamental Theorem of Calculus to $F(\sigma)= U_1(t,\sigma) U_2(\sigma,s) $ in the interval $\sigma\in [s,t]$. 
Apply $U_0(-s)\mathbf{f}$ on the right-hand,  and take the (strong) limit   $s\to-\infty$ to get
\be{pr1-3}
U_1(t,0)\Omega_{1,-}\mathbf{f}  -U_2(t,0)\Omega_{2,-}\mathbf{f}  = \int_{-\infty}^t    U_1(t,\sigma)(Q_1(\sigma)-Q_2(\sigma)) \left[U_2(\sigma,0) \Omega_{2,-}\mathbf{f}\right] \,\d \sigma.
\ee
For the left-hand side, and for the expression in the square brackets we will apply Theorem~\ref{TD_thm_sc1}(a). We take the asymptotic wave profile of  \r{pr1-3}  next applying Theorem~\ref{thm_TD_AWP2}. Then \r{pr1-est} is just that expression, written as a composition of Schwartz kernels, eventually applied to $\mathcal{R} \mathbf{f}$. 
\end{proof}

\subsection{Progressive wave expansion}\label{Prog_wave_eq}
We have the following progressive wave expansion. Let $h=:h_0$ be the Heaviside function and set $h_j (\tau) = \tau^j/j!$ for $\tau>0$; $h_j (\tau) =0$ for $\tau\le 0$. 
\begin{proposition}[\cite{MR1004174}]   \label{pr2}
For each integer $N\ge0$ we have
\be{eq_pr2}
u(t,s,x,\omega)\sim \delta(t+s-x\cdot\omega) + \sum_{j=0}^N a_j(t,x,\omega) h_j(t+s-x\cdot\omega)+R_N(t,x,s,\omega)  ,
\ee
where 
\[
\begin{split}
a_0(t,x,\omega) &= -\frac12 \int_{-\infty}^0 q(t+\tau, x+\tau\omega)\,\d\tau,\\
a_j(t,x,\omega) &= -\frac12 \int_{-\infty}^0(\Box+q) a_{j-1}(t+\tau, x+\tau\omega, \omega)\, \d\tau , \quad j=1,\dots,N,
\end{split}
\]
and $R_N\in C(\R_t\times \R_s\times S_\omega^{n-1};\; H^{N+1}(\R^n_x))$.
\end{proposition}

The latter statement follows from the fact that $R_N$ solves 
\[
(\partial_t^2-\Delta+q)R_N= - [(\partial_t^2-\Delta+q )A_N]h_N, \quad R_N|_{t<-s-\rho}=0.
\]
We get the same expansion for $u^+$ but with a different remainder $R_N$.

\subsection{Sketch of the main idea}\label{sec_idea}
Using Proposition~\ref{pr1}, and keeping the most singular terms of $u_1^-$ and $u_2^+$ only, we get
\be{P1}
\begin{split}
\delta A^\sharp(s',\omega';s,\omega) & \sim \int \delta q(t,x) \delta(t+s-x\cdot\omega) \delta(t+s'-x\cdot\omega')\,\d t\,\d x, 
\end{split}
\ee
where $\delta A^\sharp$ and $\delta q$ are formal linearizations, while the other two deltas above are Dirac deltas. 

The product of the two deltas is a delta, with the coefficient $2(4-(1+\omega\cdot\omega')^2)^{-1/2}$ 
on the $n-1$ dimensional hyperplane (co-dimension $2$) given by the system
\be{4}
-t+x\cdot\omega=s, \quad -t +x\cdot\omega'=s'
\ee
with $s$, $s'$ parameters, assuming $\omega\not=\omega'$, i.e., staying away from the forward scattering directions. Its conormal bundle is the span of $(-1,\omega)$ and $(-1,\omega')$. Those are two lightlike covectors, and all future pointing lightlike covectors  look like this. Taking linear combinations, and varying $\omega$ and $\omega'$, we get all covectors. So we are really inverting the $k=(n-1)$ -- Radon transform in $\R^{1+n}$ (by Helgason's terminology \cite{Helgason-Radon}) over \textit{all} $k$-planes; and this is stably invertible. 
We must stay away from $\omega=\omega'$ though. 
On the other hand, the codimension two Radon transform is overdetermined, so we do not need all of them, and we can avoid the bad planes. Back-scattering only ($\omega=-\omega'$) is one case where this works. 

Consider the \textbf{back-scattering} ($\omega'=-\omega$) problem now. Then \r{4} reduces to
\be{5}
-t+x\cdot\omega=s, \quad -t -x\cdot\omega=s'
\ee
which implies
\be{6}
x\cdot\omega= (s-s')/2=\sigma, \quad t =   -(s+s')/2 = -\sigma'.
\ee
This is easy to visualize as two hyperplanes in time-space at angle 45 degrees with the $t$-axis each, intersecting at a right angle, see Figure~\ref{time-dep-pot_pics1} below. 
\begin{figure}[h!] 
  \centering
  \includegraphics[scale=1,page=1]{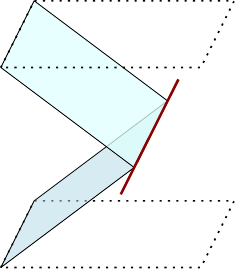}
\caption{\small  On the principal level, we integrate over the horizontal line in the middle, which is a codimension two hyperplane, actually. The support of the integrand, all terms considered, is inside the wedge, intersected with the cylinder $|x|\le\rho$.  
}
\label{time-dep-pot_pics1}
\end{figure}
Varying $\omega$, and $s, s'$ so that $s+s'$ is fixed, we get integrals over all ``lines'' (codimension two hyperplanes, actually) on the hyperplane $t=\text{const.}$, which is invertible, slice by slice. The stability estimate we get however treats $t$ and $x$ differently, in principle.

We need to analyze the lower order terms now. If we linearize near $q\not=0$, and take the lower order term $a_0$ for $u_1^-$ in \eqref{eq_pr2} into account only, then \r{P1} is the main term but we get the following additional terms:
\[
\begin{split}
B_1(s',\omega';s,\omega) &:= \int \delta q(t,x) a_1(t,x,\omega) h(t+s-x\cdot\omega) \delta(t+s'-x\cdot\omega')\,\d t\,\d x, 
\end{split}
\]
plus a similar $B_2$ term coming from $u_1^+$, plus an even more regular term containing two Heaviside functions $h$.  Each of  $B_1$, $B_2$ integrates $\delta q$ over a (weighted)  truncated delta on the hyperplane $t+s'-x\cdot\omega'=0$, or $t+s-x\cdot\omega=0$, respectively, see Figure~\ref{time-dep-pot_pics1}. In other words, the Schwartz kernels are deltas on half-hyperplanes. They can be thought of as a superposition of deltas on codimension two hyperplanes, as in \r{P1}, all parallel to the intersection in Figure~\ref{time-dep-pot_pics1}, moving along the corresponding wing of the wedge there. We just need upper bounds of $B_1$, $B_2$, and they can be done in the same norms as those we use for \r{P1} to get an $O(\varepsilon\|q\|)$ perturbation. The $\varepsilon$ gain comes from $a_1$. We can treat the more regular terms, coming from $a_j$, $j\ge1$, and from $R_N$ in Proposition~\ref{pr2} similarly, which is a tedious task but doable. An important observation is that the support of the product $u_1^- u_2^+$ in \r{pr1-est} is contained in the intersection of the half-space above the lower hyperplane (due to $u_1^-$, see \r{FS}) and the half-space below the upper hyperplane (due to $u_1^+$), i.e., on the left of that the wedge in the figure. This is further intersected with the cylinder $|x|\le\rho$, so in particular, we integrate in \r{pr1-est} over a compact set depending on the parameters.

\subsection{Proof of the main result}
Here, we will prove our main result. We begin with a few notations. Inspired by \r{pr1-est}, we set
\begin{equation*}
    q(t,x) = q_1(t,x) - q_2(t,x)
\end{equation*}
and
\be{2.1}
Mq (s', s,\omega) = \int q(t ,x )u_1^-(t ,x,s,\omega)u_2^+(t,x,s',-\omega )\, \d t\,\d x.
\ee
Writing $u_1^- (t ,x,s,\omega) = \delta(t+s- x\cdot\omega)+u_{1,\rm sc}^-$, $u_2(t ,x,s,-\omega) = \delta(t+s+ x\cdot\omega)+u_{2,\rm sc}^+$, we get
\be{2.2}
M = M_{00}+M_{01}+ M_{10}+ M_{11},
\ee
where

\begin{align}
M_{00}q(s',s,\omega) &= \int  q(t,x) \delta(t+s-x\cdot\omega) \delta(t+s'+x\cdot\omega)\,\d t\,\d x, \label{M0}\\
M_{10}q(s',s,\omega) &= \int  q(t,x)  u_{1,\rm sc}^- (t,x,s,\omega) \delta(t+s'+x\cdot\omega)\,\d t\,\d x,\label{M10}\\
M_{01}q(s',s,\omega) &= \int  q(t,x) \delta(t+s-x\cdot\omega) u_{2,\rm sc}^+ (t, x,s', -\omega)\,\d t\,\d x,\label{M01}\\
M_{11}q(s',s,\omega) &= \int  q(t,x) u_{1,\rm sc}^-(t,x,s,\omega) u_{2,\rm sc}^+(t,x, s', -\omega)\,\d t\,\d x. \label{M11}
\end{align}
We denote by $\tilde{M}q$, $\tilde{M}_{kl}q$ the above functions in the variables given by \eqref{change_var}. 

\begin{remark}\label{rem_Sigma_T}
    Let 
    \be{Sigma}
    \Sigma = [-\rho, \rho].
    \ee
    Note that if $\sigma \notin \Sigma$, then \eqref{6} does not hold for any $x\in \Omega$ and $\omega\in \mathbb{S}^{n-1}$, and hence, the line of intersection (a hyperplane) in Figure \ref{time-dep-pot_pics1} does not intersect the cylinder $\mathbb{R}\times \Omega$. Therefore, $\tilde{M}_{00}q(\sigma',\sigma,\omega)= 0$ for $\sigma \notin \Sigma$. If $\sigma\in \Sigma$, due to the definitions of $u_{1,sc}^+$ and $u_{2,sc}^+$, it follows that there exists $T>0$ such that the $(t,x)$-supports of the integrands in the definition of $\tilde{M}q$, $\tilde{M}_{kl}q$ belongs to $(-T - \sigma',T - \sigma)\times \Omega$.
\end{remark}

Next, we will study $\tilde{M}_{kl}q$. We begin with $\tilde{M}_{00}q$:

\begin{lemma}\label{est_M00}
    Let $\Sigma = [-\rho, \rho]$. There exists a constant $C>0$ depending only on $n$ such that
    \begin{equation}\label{est_M00_0}
        \|q\|_{L^\infty(\R;\; L^2(\R^n))} / C\leq \|\tilde{M}_{00}q\|_{L^\infty(\mathbb{R}_{\sigma'}; L^2(\mathbb{S}_\omega^{n-1}; H^{(n-1)/2}(\Sigma_\sigma)))} \leq C \|q\|_{L^\infty(\R;\; L^2(\R^n))}.
    \end{equation}
\end{lemma}
\begin{proof}
    As explained in Section \ref{sec_idea}, the product of the deltas in \r{M0} can also be written as 
    \begin{equation*}
        \delta\left(t+ \frac{s+s'}{2}\right)\delta\left(\frac{s-s'}{2} - x\cdot\omega\right).
    \end{equation*}
    Then
\[
M_{00}q(s',s,\omega) = \left[Rq(-(s'+s)/2, \cdot)\right] ((s-s')/2,\omega),
\]
where $Rf(p,\omega)= \int\delta(p-x\cdot\omega) f(x)\,\d x$ is the Radon transform of $f$. Using the change of variables given by \eqref{change_var}, we rewrite
\[
\tilde M_{00}q(\sigma',\sigma,\omega) = \left[Rq(-\sigma', \cdot)\right] (\sigma,\omega).
\]
For every $s$, we have the following stability estimate for the Radon transform, see \cite[Ch.~2]{Natterer-book}:
\[
\|f\|_{H^s(\R^n)}/C \le \|Rf\|_{L^2(\mathbb{S}_\omega^{n-1}; H^{s + (n-1)/2}(\mathbb{R}_\sigma))} \le C\|f\|_{H^s(\R^n)},
\]
for some constant $C>0$. Therefore, choosing $s=0$, we obtain
\be{2.3a}
\|q(-\sigma',\cdot)\|_{L^2(\R^n)}/C \le \|\tilde M_{00} q(\sigma',\cdot,\cdot)\|_{L^2(\mathbb{S}_\omega^{n-1}; H^{(n-1)/2}(\mathbb{R}_\sigma))} \le C\|q(-\sigma',\cdot)\|_{L^2(\R^n)}
\ee
for every $\sigma'$ with $C$ independent of it. Finally, as we noted in Remark~\ref{rem_Sigma_T},
\begin{equation*}
    \|\tilde{M}_{00}q\|_{L^\infty(\mathbb{R}_{\sigma'}; L^2(\mathbb{S}_\omega^{n-1}; H^{(n-1)/2}(\Sigma_\sigma)))} = \|\tilde{M}_{00}q\|_{L^\infty(\mathbb{R}_{\sigma'}; L^2(\mathbb{S}_\omega^{n-1}; H^{(n-1)/2}(\mathbb{R}_\sigma)))},
\end{equation*}
and hence, \eqref{est_M00_0} holds. 
\end{proof}

We need the next lemmas to obtain similar results for $\tilde{M}_{10}q$ and $\tilde{M}_{01}q$.

\begin{lemma}\label{M_10_interms_aR}
    Let $\Sigma$ and $T$ be as in Remark~\ref{rem_Sigma_T}.  Let $a_{1,j}$, $h_j$, and $R_{1,N}$ be the functions introduced in Section \ref{Prog_wave_eq} with $q=q_1$ and let $\bar{R}_{1,N}$ be the function such that
    \begin{equation*}
        \bar{R}_{1,N} (t,x, t + s - x\cdot\omega,\omega) = R_{1,N}(t,x,s,\omega).
    \end{equation*}
    Then, there exists $C_{\Omega,N}>0$ such that
     \begin{multline*}
    \|\tilde{M}_{10}q\|_{L^\infty(\mathbb{R}_{\sigma'};L^2(\mathbb{S}_\omega^{n-1}; H^{(n-1)/2}(\Sigma_\sigma)))} \leq C_{\Omega,N} \|q\|_{L^\infty(\mathbb{R}; L^2(\mathbb{R}^n))}\times\\       
    \times\left(\sum_{j=0}^N  \|a_{1,j} \|_{L^\infty(\mathbb{R}\times \mathbb{S}^{n-1};C^{2n-2}(\bar\Omega))} 
        + \sup_{\sigma'\in \mathbb{R}}\int_{-T-\sigma'}^{T-\sigma'} \sup_{\omega \in \mathbb{S}^{n-1}}\|\bar{R}_{1,N} (t,\cdot, 2t + 2\sigma',\omega)\|_{C^{2n-2}(\bar\Omega)}d t\right).
\end{multline*}
\end{lemma}
\begin{proof}
    For the sake of brevity, we denote
    \begin{equation*}
        U(t,x, t + s - x\cdot\omega,\omega) = \sum_{j=0}^N a_{1,j}(t,x,\omega) h_j(t+s-x\cdot\omega) + \bar{R}_{1,N} (t,x, t + s - x\cdot\omega,\omega).
    \end{equation*}
    Then,
    \begin{align*}
        M_{10}q(s',s,\omega) &= \int_{\mathbb{R}^n} \int_{\mathbb{R}} q(t,x)  U (t,x, t + s - x\cdot\omega,\omega) \delta(t+s'+x\cdot\omega)\,\d t\,\d x\\
        & = \int_{\mathbb{R}} \int_{\omega^\perp} q(t,-(t + s') \omega + y)  U(t,-(t + s') \omega + y, 2t + s + s',\omega) \,\d t\,\d y.
    \end{align*}
    Using the change of variables given by \eqref{change_var}, for any $\sigma\in \Sigma$, we obtain
\begin{align*}
    \tilde{M}_{10}q(\sigma',\sigma,\omega) &
    = \int_{-T - \sigma'}^{T - \sigma'} \int_{\omega^\perp} q(t,-(t + \sigma' - \sigma) \omega + y)  U(t,-(t + \sigma' - \sigma) \omega + y, 2t + 2\sigma',\omega) \,\d t\,\d y\\&=\int_{-T - \sigma'}^{T - \sigma'} \left[R_{U(t,\cdot,2t + 2\sigma',\cdot)}q(t,\cdot)\right] (-t + \sigma- \sigma',\omega)\,\d t.
\end{align*}
By Theorem \ref{w_Radon_t_est},
\begin{multline*}
        \|\tilde{M}_{10}q(\sigma',\cdot,\cdot)\|_{L^2(\mathbb{S}_\omega^{n-1}; H^{(n-1)/2}(\Sigma_\sigma))}\\
        \leq C_\Omega \sup_{t\in [-T - \sigma',T-\sigma']}\|q(t,\cdot)\|_{L^2(\Omega)}  \int_{-T-\sigma'}^{T-\sigma'} \sup_{\omega \in \mathbb{S}^{n-1}}\|U (t,\cdot, 2t + 2\sigma',\omega)\|_{C^{2n-2}(\bar\Omega)}d t.
    \end{multline*}
    We estimate next
    \begin{multline}
        \int_{-T-\sigma'}^{T-\sigma'} \sup_{\omega \in \mathbb{S}^{n-1}}\|U (t,\cdot, 2t + 2\sigma',\omega)\|_{C^{2n-2}(\bar\Omega)}d t\leq C_N \int_{-T-\sigma'}^{T-\sigma'} \sup_{\omega \in \mathbb{S}^{n-1}}\|a (t,\cdot,\omega)\|_{C^{2n-2}(\bar\Omega)}d t \\
        + \int_{-T-\sigma'}^{T-\sigma'} \sup_{\omega \in \mathbb{S}^{n-1}}\|\bar{R}_{1,N} (t,\cdot, 2t + 2\sigma',\omega)\|_{C^{2n-2}(\bar\Omega)}d t.
    \end{multline}
    The last two estimates complete the proof.
\end{proof}

\begin{lemma}\label{energ_est}
    Let $\Omega\subset \mathbb{R}^n$ be a bounded set, $\alpha$ be a multi-index, and $K = |\alpha| +(n+1)/2$. Let $f\in C^N(\mathbb{R}_t\times \mathbb{R}_s\times \mathbb{R}^n_x)$ such that $f(t,s,\cdot)$ has a compact support. Assume that $N\in\mathbb{N}$ is large enough so that the equation
    \begin{equation*}
    \begin{cases}
        (\partial_t^2 - \Delta + q(t,x))v(t,s,x) = f(t,s,x),\\
        v\arrowvert_{t<-s-\rho} = 0.
    \end{cases}
    \end{equation*}
    has the unique smooth (as much as needed) solution $v$. Then
    \begin{equation}\label{energ_est_0}
        \|D_x^\alpha v\|_{L^\infty(\Omega)} \leq C_{\Omega,K} e^{C_q^K(t+s+\rho)} \int_{-s-\rho}^t \|f(\tau,s,\cdot)\|_{H^K(\mathbb{R}^n)}d\tau,
    \end{equation}
    where 
    \begin{equation*}
        C_q^K = 2 + C_K\|q(t,\cdot)\|_{C^k(\mathbb{R}^n)}
    \end{equation*}
    with some $C_K>0$ dependent on $K$. 
\end{lemma}

\begin{proof}
    For non-negative $k\in \mathbb{Z}$, we define
    \begin{equation*}
        E_s^k(t) = \|\partial_tv(t,s,\cdot)\|_{H^k(\mathbb{R}^n)}^2 + \sum_{|\gamma|\leq k} \|\nabla_x D_x^\gamma v(t,s,\cdot)\|_{(L^2(\mathbb{R}^2))^n}^2 + \|v(t,s,\cdot)\|_{L^2(\mathbb{R}^2)}^2.
    \end{equation*}
    Then 
    \begin{multline*}
        \partial_t E_s^k(t) = 2 \sum_{|\gamma|\leq  k} \Re(\partial_t^2 D_x^\gamma v(t,s,\cdot),  \partial_t D_x^\gamma v(t,s,\cdot))_{L^2(\mathbb{R}^2)} \\
        - 2 \sum_{|\gamma|\leq  k} \Re(\Delta D_x^\gamma v(t,s,\cdot), \partial_t D_x^\gamma v(t,s,\cdot) )_{L^2(\mathbb{R}^2)} + 2\Re(\partial_t v(t,s,\cdot), v(t,s,\cdot))_{L^2(\mathbb{R}^2)},
    \end{multline*}
    where $\Re$ denotes the real part of the number. Since 
    \begin{equation*}
        (\partial_t^2 - \Delta)D_x^\gamma v(t,s,x) + D_x^\gamma (q(t,x)v(t,s,x)) = D_x^\gamma f(t,s,x),
    \end{equation*}
    the last identity becomes
    \begin{multline*}
        \partial_t E_s^k(t) = 2 \sum_{|\gamma|\leq  k} \Re(D_x^\gamma f(t,s,\cdot),  \partial_t D_x^\gamma v(t,s,\cdot))_{L^2(\mathbb{R}^2)}  \\
        - 2 \sum_{|\gamma|\leq  k}\Re(D_x^\gamma (q(t,\cdot) v(t,s,\cdot)), \partial_tD_x^\gamma v(t,s,\cdot) )_{L^2(\mathbb{R}^2)} + 2\Re(\partial_t v(t,s,\cdot), v(t,s,\cdot))_{L^2(\mathbb{R}^2)}.
    \end{multline*}
    Using the Cauchy–Schwarz inequality, we obtain
    \begin{equation*}
        \partial_t E_s^k(t) \leq \sum_{|\gamma|\leq  k} \|D_x^\gamma f(t,s,\cdot)\|_{L^2(\mathbb{R}^n)}^2 + (2 + C_k\|q(t,\cdot)\|_{C^k(\mathbb{R}^2)} ) E_s^k(t).
    \end{equation*}
    By integration over $(-s-\rho, t)$, we obtain
    \begin{equation*}
        E_s^k(t) \leq F_s^k(t) + C_q^k \int_{-s-\rho}^t E_s^k(\tau) d\tau,
    \end{equation*}
    where
    \begin{equation*}
        F_s^k(t) = \int_{-s-\rho}^t \|f(\tau,s,\cdot)\|_{H^k(\mathbb{R}^n)}d\tau,
        \qquad
        C_q^k = 2 + C_k\|q(t,\cdot)\|_{C^k(\mathbb{R}^n)}.
    \end{equation*}
    Then, the Gr\"{o}nwall's inequality implies
    \begin{equation*}
        E_s^k(t) \leq F_s^k(t) + C_q^k \int_{-s-\rho}^t F_s^k(\tau) e^{C_q^k(t - \tau)}d\tau.
    \end{equation*}
    Since $F_s^k$ is an increasing function,
    \begin{equation}\label{energ_est_1}
        E_s^k(t) \leq F_s^k(t)\left( 1 + C_q^k \int_{-s-\rho}^t e^{C_q^k(t - \tau)}d\tau\right) \leq F_s^k(t) e^{C_q^k(t+s+\rho)}.
    \end{equation}
    Due to the Sobolev inequality, it follows that 
    \begin{equation*}
        \|v(t,s,\cdot)\|_{C^{|\alpha|}(\bar\Omega)} \leq C_{\Omega,K} \|v(t,s,\cdot)\|_{H^{K+1}(\Omega)} \leq C_{\Omega,K} \|v(t,s,\cdot)\|_{H^{K+1}(\mathbb{R}^n)}.
    \end{equation*}
    Combining this with \eqref{energ_est_1}, we obtain \eqref{energ_est_0}.    
\end{proof}
\begin{lemma}\label{est_for_aR}
    Let $L = 2n - 2$. There exists a sufficiently large $N\in\mathbb{R}$ such that if
    \begin{equation}\label{est_for_aR_0}
        \|q_1\|_{C^{L + \frac{n-1}{2}+3+2N}(\mathbb{R}\times\bar\Omega)}<1.
    \end{equation}
    Then,
    \begin{equation*}
        \sum_{j=0}^N  \|a_{1,j} \|_{L^\infty(\mathbb{R}\times \mathbb{S}^{n-1};C^{2n-2}(\bar\Omega))} \leq C_{\Omega,N} \|q_1\|_{C^{L + 2N}(\mathbb{R}\times\bar\Omega)}
    \end{equation*}
    and
    \begin{equation*}
        \sup_{\sigma'\in \mathbb{R}}\int_{-T-\sigma'}^{T-\sigma'} \sup_{\omega \in \mathbb{S}^{n-1}}\|\bar{R}_{1,N} (t,\cdot, 2t + 2\sigma',\omega)\|_{C^{L} (\bar\Omega)}d t \leq C_{\Omega,N} \|q_1\|_{C^{L + \frac{n-1}{2}+3+2N}(\mathbb{R}\times \bar \Omega)}.
    \end{equation*}
\end{lemma}
\begin{proof}
    The first estimate comes directly from the definition of $a_{1,j}$. It remains to show the second estimate. Set
    \begin{equation*}
        A = \sup_{\sigma'\in \mathbb{R}}\int_{-T-\sigma'}^{T-\sigma'} \sup_{\omega \in \mathbb{S}^{n-1}}\|\bar{R}_{1,N} (t,\cdot, 2t + 2\sigma',\omega)\|_{C^{L}(\bar\Omega)}d t.
    \end{equation*}
    We write
    \begin{align*}
        A &= \sup_{\sigma'\in \mathbb{R}}\int_{-T}^{T} \sup_{\omega \in \mathbb{S}^{n-1}}\|\bar{R}_{1,N} (t-\sigma',\cdot, 2t ,\omega)\|_{C^{L}(\bar\Omega)}d t \\
        & = \sup_{\sigma'\in \mathbb{R}}\int_{-T}^{T} \sup_{\omega \in \mathbb{S}^{n-1}} \sum_{|\gamma|\leq L}\sup_{x\in \Omega} \left|[D_x^\gamma\bar{R}_{1,N}](t-\sigma',x, 2t,\omega )\right|dt\\
        & = \sup_{\sigma'\in \mathbb{R}}\int_{-T}^{T} \sup_{\omega \in \mathbb{S}^{n-1}} \sum_{|\alpha|+|\beta|\leq L}\sup_{x\in \Omega} \left|[D_x^{\alpha}D_{s}^{|\beta|}R_{1,N}](t-\sigma',x, 2t - (t-\sigma') + x\cdot\omega,\omega)\right|dt.
    \end{align*}
    Then, we estimate 
    \begin{align}\label{est_for_A}
        \nonumber A & \leq \sum_{|\alpha|+|\beta|\leq L}2T\sup_{\sigma'\in \mathbb{R}} \sup_{\omega \in \mathbb{S}^{n-1}} \sup_{x\in \Omega} \sup_{t\in [-T,T]} \left|[D_x^{\alpha}D_{s}^{|\beta|}R_{1,N}](t-\sigma',x, 2t - (t-\sigma') + x\cdot\omega,\omega)\right|\\
        &\leq \sum_{|\alpha|+|\beta|\leq L}2T\sup_{t\in \mathbb{R}} \sup_{\omega \in \mathbb{S}^{n-1}} \sup_{x\in \Omega} \sup_{s\in [-2T,2T]} \left|[D_x^{\alpha}D_{s}^{|\beta|}R_{1,N}](t,x, s - t + x\cdot\omega,\omega)\right|.
    \end{align}
    To estimate the last term, let us fix multi-indexes $\alpha$, $\beta$ such that $|\alpha| + |\beta| = L$ and note that
    \begin{equation*}
        (\partial_t^2-\Delta+q)D_{s}^{|\beta|}R_N(t,x,s,\omega)= - [(\partial_t^2-\Delta+q )a_{1,N}](t,x,\omega)h_{N}^{(|\beta|)}(s + t - x\cdot\omega).
    \end{equation*}
    Moreover, by employing the derivative definition, it can be verified that
    \begin{equation*}
        D_{s}^{|\beta|}R_N(t,x,s,\omega) \arrowvert_{t < -s - \rho} = 0.
    \end{equation*}
    We denote
    \begin{equation}\label{A1N}
        A_{1,N}(\tau,x,\omega) = -[(\partial_t^2-\Delta+q )a_{1,N}](t,x,\omega),
    \end{equation}
    Next, we will show that 
    \begin{equation}\label{source}
        f(t,s,x,\omega) = A_{1,N}(\tau,x,\omega)h_{N}^{(|\beta|)}(s + t - x\cdot\omega)
    \end{equation}
    has a compact support as a function of the $x$ variable. In \cite{MR1004174}, it was shown that 
    \begin{equation*}
        a_{1,N}(t,x,\omega)  = 0, 
        \quad \text{for } x\cdot\omega <-\rho
        \text{ and for } |x - x\cdot\omega|>\rho.
    \end{equation*}
    Hence, if $x\cdot\omega\leq 0$, for sufficiently large $|x|$, it follows that $a_{1,N}(t,x,\omega)  = 0$. If $x\cdot\omega > 0$, then $h_{N}^{(|\beta|)}(s + t - x\cdot\omega) = 0$ for sufficiently large $|x|$. Therefore, for fixed $t$, $s$, and $\omega$, the function given by \eqref{source} is compactly supported. Therefore, by Lemma \ref{energ_est},
    \begin{multline*}
        \sup_{x\in \Omega}|D^\alpha_xD_{s}^{|\beta|}R_{1,N}(t,x,s,\omega) |\\
        \leq C_{\Omega,K} e^{C_q^K(t+s+\rho)} \int_{-s-\rho}^t \|A_{1,N}(\tau,\cdot,\omega)h_{N}^{(|\beta|)}(s + \tau - (\cdot)\cdot\omega)\|_{H^K(\mathbb{R}^n)}d\tau,
    \end{multline*}
    where
    \begin{equation}\label{def_K}
        K = |\alpha| + \frac{n-1}{2} +1,
        \qquad
        C_q^K = 2 + C_K\|q(t,\cdot)\|_{C^K(\mathbb{R}^n)}.
    \end{equation}
    Since $h_{N}^{(|\beta|+k)}$ is an increasing function for all $k=0,\cdots,K$, it follows that
    \begin{multline*}
        \sup_{x\in \Omega}|D^\alpha_xD_{s}^{|\beta|}R_{1,N}(t,x,s,\omega) | \leq  C_{\Omega,K} e^{C_q^K(t+s+\rho)} (t + s +\rho) 
        \\
        \times \sup_{\tau\in\mathbb{R}}\|A_{1,N}(\tau,\cdot,\omega)h_{N}^{(|\beta|)}(s + t - (\cdot)\cdot\omega)\|_{H^K(\mathbb{R}^n)}.
    \end{multline*}
    This is true for all $s\in \mathbb{R}$. Hence, if we choose $y\in \Omega$ and replace $s$ by $s - t +y\cdot\omega$, the last estimate becomes
    \begin{multline*}
        \sup_{x\in \Omega}|[D^\alpha_xD_{s}^{|\beta|}R_{1,N}](t,x,s - t +y\cdot\omega,\omega) | \leq  C_{\Omega,K} e^{C_q^K(s+y\cdot\omega+\rho)} (s+y\cdot\omega+\rho) 
        \\
        \times \sup_{\tau\in\mathbb{R}}\|A_{1,N}(\tau,\cdot,\omega)h_{N}^{(|\beta|)}(s + y\cdot\omega - (\cdot)\cdot\omega)\|_{H^K(\mathbb{R}^n)}.
    \end{multline*}
   Let
    \begin{equation*}
        z = \sup_{s\in [-2T,2T]} \sup_{y\in \Omega} \sup_{\omega\in \mathbb{S}^{n-1}} (s + y\cdot\omega).
    \end{equation*}
    Note that $z$ is constant depending only on $\Omega$. Then, from the last inequality, we derive
    \begin{multline}\label{3sup_est}
        \sup_{x,y\in \Omega}\sup_{t\in \mathbb{R}}\sup_{s\in [-2T,2T]}|[D^\alpha_xD_{s}^{|\beta|}R_{1,N}](t,x,s - t +y\cdot\omega,\omega) | \leq  C_{\Omega,K} e^{C_q^K(z+\rho)} (z+\rho) 
        \\
        \times \sup_{\tau\in\mathbb{R}}\|A_{1,N}(\tau,\cdot,\omega)h_{N}^{(|\beta|)}(z - (\cdot)\cdot\omega)\|_{H^K(\mathbb{R}^n)}.
    \end{multline}
    Moreover, we know that
    \begin{equation*}
        A_{1,N}(\tau,\cdot,\omega)h_{N}^{(|\beta|)}(z - (\cdot)\cdot\omega)
    \end{equation*}
    has a compact support with respect to $x$, which is uniformly bounded in $\tau \in \mathbb{R}$ and $\omega\in \mathbb{S}^{n-1}$, that is, there is a compact set $\tilde{\Omega}$  such that
    \begin{equation*}
        \supp A_{1,N}(\tau,\cdot,\omega)h_{N}^{(|\beta|)}(z - (\cdot)\cdot\omega) \subset \tilde{\Omega}
        \qquad
        \text{for all } \tau\in \mathbb{R} \text{ and }\omega\in \mathbb{S}^{n-1}.
    \end{equation*}
    The set $\tilde\Omega$ depends only on $\Omega$. Therefore,
    \begin{multline}\label{est_A1N_hN}
        \sup_{\tau\in\mathbb{R}}\|A_{1,N}(\tau,\cdot,\omega)h_{N}^{(|\beta|)}(z - (\cdot)\cdot\omega)\|_{H^K(\mathbb{R}^n)} \leq \sup_{\tau\in\mathbb{R}}\|A_{1,N}(\tau,\cdot,\omega)h_{N}^{(|\beta|)}(z - (\cdot)\cdot\omega)\|_{H^K(\tilde{\Omega})} \\
         \leq C_{\Omega,N} \sup_{\tau\in\mathbb{R}}\|A_{1,N}(\tau,\cdot,\omega)\|_{H^K(\tilde{\Omega})}.
    \end{multline}
    Therefore, since
    \begin{multline*}
        \sup_{x\in \Omega}\sup_{t\in \mathbb{R}}\sup_{s\in [-2T,2T]}|[D^\alpha_xD_{s}^{|\beta|}R_{1,N}](t,x,s - t +x\cdot\omega,\omega) |\\
        \leq \sup_{x,y\in \Omega}\sup_{t\in \mathbb{R}}\sup_{s\in [-2T,2T]}|[D^\alpha_xD_{s}^{|\beta|}R_{1,N}](t,x,s - t +y\cdot\omega,\omega) |,
    \end{multline*}
    from \eqref{3sup_est}, \eqref{def_K}, and \eqref{est_for_aR_0}, we obtain
    \begin{multline*}
        \sup_{x\in \Omega}\sup_{t\in \mathbb{R}}\sup_{s\in [-2T,2T]}\sup_{\omega\in \mathbb{S}^{n-1}}|[D^\alpha_xD_{s}^{|\beta|}R_{1,N}](t,x,s - t +x\cdot\omega,\omega) |\\
        \leq C_{\Omega,N} \sup_{\omega\in \mathbb{S}^{n-1}}\sup_{\tau\in\mathbb{R}}\|A_{1,N}(\tau,\cdot,\omega)\|_{C^{K}(\tilde{\Omega})}.
    \end{multline*}
    Since
    \begin{equation}\label{est_A1N}
        \sup_{\omega\in \mathbb{S}^{n-1}}\sup_{\tau\in\mathbb{R}}\|A_{1,N}(\tau,\cdot,\omega)\|_{C^K(\tilde{\Omega})} \leq C_{\Omega,N} \|q_1\|_{C^{K+2+2N}(\mathbb{R}\times\bar\Omega)},
    \end{equation}
    the previous estimate implies
    \begin{equation*}
        \sup_{x\in \Omega}\sup_{t\in \mathbb{R}}\sup_{s\in [-2T,2T]}\sup_{\omega\in \mathbb{S}^{n-1}}|[D^\alpha_xD_{s}^{|\beta|}R_{1,N}](t,x,s - t +x\cdot\omega,\omega) |
        \leq C_{\Omega,N} \|q_1\|_{C^{K+2+2N}(\mathbb{R}\times\bar\Omega)}.
    \end{equation*}
    Then, from \eqref{est_for_A}, it follows that
    \begin{equation*}
        A \leq C_{\Omega,N} \|q_1\|_{C^{L + \frac{n-1}{2}+3+2N}(\mathbb{R}\times\bar\Omega)}.
    \end{equation*}
    This completes the proof. 
\end{proof}

Now, we are ready to estimate $M_{10}q$. Similarly, the same holds for $M_{01}q$.
\begin{lemma}\label{est_M10}
    Let $\Sigma$ be as in \r{Sigma}. There exists a sufficiently large $k\in\mathbb{N}$ such that if
    \begin{equation}\label{est_M10_0}
        \|q_1\|_{C^{k}(\mathbb{R}\times \bar\Omega)} \leq \varepsilon < 1,
    \end{equation}
    then
    \begin{equation}\label{est_M10_1}
         \|\tilde{M}_{10}q\|_{L^\infty(\mathbb{R}_{\sigma'};L^2(\mathbb{S}_\omega^{n-1}; H^{(n-1)/2}(\Sigma_\sigma)))} \leq \varepsilon C_{\Omega} \|q\|_{L^\infty(\mathbb{R}; L^2(\mathbb{R}^n))}.
    \end{equation}
\end{lemma}
\begin{proof}
    Depending on $\Omega$, we choose $N\in\mathbb{N}$ large enough so that the hypothesis of Lemma \ref{est_for_aR} is satisfied. Let $k=L + \frac{n-1}{2}+3+2N$. Then, Lemma \ref{est_for_aR} and \eqref{est_M10_0} give \eqref{est_M10_1}. 
\end{proof}

Next, we obtain a similar result for $M_{11}q$.
\begin{lemma}\label{est_M11}
    Let $\Sigma$ be as in \r{Sigma}. There exists a sufficiently large $k\in\mathbb{R}$ such that if
    \begin{equation}\label{est_M11_0}
        \|q_1\|_{C^{k}(\mathbb{R}\times \bar\Omega)} \leq \varepsilon < 1,
    \end{equation}
    then
    \begin{equation}\label{est_M11_1}
         \|\tilde{M}_{11}q\|_{L^\infty(\mathbb{R}_{\sigma'};L^2(\mathbb{S}_\omega^{n-1}; H^{(n-1)/2}(\Sigma_\sigma)))} \leq \varepsilon C_{\Omega} \|q\|_{L^\infty(\mathbb{R}; L^2(\mathbb{R}^n))}.
    \end{equation}
\end{lemma}
\begin{proof}
    Let $N\in \mathbb{N}$ be sufficiently large. We set
    \begin{align*}
        &Q_j^k(t,x,\omega) = q(t,x) a_{1,j}(t,x,\omega)a_{2,k}(t,x,-\omega),\\
        &Q_j(t,x,\omega) = q(t,x) a_{1,j}(t,x,\omega),\\
        &Q^k(t,x,\omega) = q(t,x) a_{2,k}(t,x,-\omega).
    \end{align*}
    Next, we define
\begin{align*}
    &A_{kj}^l(\sigma',\sigma,\omega) = \int_{\Omega}\int_{-T-\sigma'}^{T - \sigma'}Q_j^k(t,x,\omega) h_j^{(l)}(t + \sigma' + \sigma - x\cdot\omega) \delta(t + \sigma' - \sigma + x\cdot\omega) dtdx, \\
    &B_{kj}^l(\sigma',\sigma,\omega) = \int_{\Omega}\int_{-T-\sigma'}^{T - \sigma'}Q_j^k(t,x,\omega) \delta(t + \sigma' + \sigma - x\cdot\omega) h_k^{(l)}(t + \sigma' - \sigma + x\cdot\omega) dtdx, \\
    &A_k^l(\sigma',\sigma,\omega) =  \int_{\Omega}\int_{-T-\sigma'}^{T - \sigma'}Q^k(t,x,\omega) \partial_{\sigma}^{l}R_{1,N}(t,x,\sigma' + \sigma,\omega) \delta(t + \sigma' - \sigma + x\cdot\omega) dtdx,\\
    &B_j^l(\sigma',\sigma,\omega) =  \int_{\Omega}\int_{-T-\sigma'}^{T - \sigma'}Q_j(t,x,\omega) \delta(t + \sigma' + \sigma - x\cdot\omega) \partial_{\sigma}^{l} R_{2,N}(t,x,\sigma'-\sigma,\omega)  dtdx
\end{align*}
and
\begin{align*}
    &A_j^{l_1l_2}(\sigma',\sigma,\omega) =  \int_{\Omega}\int_{-T-\sigma'}^{T - \sigma'}Q_j(t,x,\omega) h_j^{(l_1)}(t + \sigma' + \sigma - x\cdot\omega) \partial_{\sigma}^{l_2} R_{2,N}(t,x,\sigma'-\sigma,\omega) dtdx, \\
    &B_{k}^{l_1l_2}(\sigma',\sigma,\omega) =  \int_{\Omega}\int_{-T-\sigma'}^{T - \sigma'}Q^k(t,x,\omega) \partial_{\sigma}^{l_1}R_{1,N}(t,x,\sigma' + \sigma,\omega) h_k^{(l_2)}(t + \sigma' - \sigma + x\cdot\omega) dtdx, \\
    &C^{l_1l_2} (\sigma',\sigma,\omega) = \int_{\Omega}\int_{-T-\sigma'}^{T + \sigma'}q(t,x) \partial_{\sigma}^{l_1}R_{1,N}(t,x,\sigma' + \sigma,\omega) \partial_{\sigma}^{l_2} R_{2,N}(t,x,\sigma'-\sigma,\omega) dtdx.
\end{align*}
Then,
\begin{equation}\label{est_M11_5}
    \|\tilde M_{11}q(\sigma',\cdot,\cdot)\|_{L^2(\mathbb{S}^{n-1}_\omega; H^{(n-1)/2}(\Sigma_\sigma))} \leq \sum M_{kj}^{l_1l_2},
\end{equation}
where the sum is taking over all $j$, $k$, $l_1$, $l_2\leq N$ such that
\begin{equation*}
    j,k\leq N, 
    \quad
    l_1<j,
    \quad
    l_2<k,
    \quad
    l_1+l_2\leq \frac{n-1}{2}
\end{equation*}
and  
\begin{align*}
    M_{kj}^{l_1l_2} = & \| \partial_\sigma^{l_1}A_{kj}^{l_2}(\sigma',\cdot,\cdot)\|_{L^2(\Sigma_\sigma\times\mathbb{S}_\omega^{n-1})} + \| \partial_\sigma^{l_1}B_{kj}^{l_2}(\sigma',\cdot,\cdot)\|_{L^2(\Sigma_\sigma\times\mathbb{S}_\omega^{n-1})}\\
    & + \| \partial_\sigma^{l_1}A_{k}^{l_2}(\sigma',\cdot,\cdot)\|_{L^2(\Sigma_\sigma\times\mathbb{S}_\omega^{n-1})} + \| \partial_\sigma^{l_1}B_{j}^{l_2}(\sigma',\cdot,\cdot)\|_{L^2(\Sigma_\sigma\times\mathbb{S}_\omega^{n-1})}\\
    & + \| A_{j}^{l_1l_2}(\sigma',\cdot,\cdot)\|_{L^2(\Sigma_\sigma\times\mathbb{S}_\omega^{n-1})} + \| B_{k}^{l_1l_2}(\sigma',\cdot,\cdot)\|_{L^2(\Sigma_\sigma\times\mathbb{S}_\omega^{n-1})} + \| C^{l_1l_2}(\sigma',\cdot,\cdot)\|_{L^2(\Sigma_\sigma\times\mathbb{S}_\omega^{n-1})} .
\end{align*}
As we noted in the proof of Lemma \ref{est_for_aR}, $\partial_s^{l_1}R_{1,N}$ satisfies
    \begin{equation}\label{est_M11_2}
        (\partial_t^2-\Delta+q)D_{s}^{l_1}R_N(t,x,s,\omega)= - [(\partial_t^2-\Delta+q )a_{1,N}](t,x,\omega)h_{N}^{l_1}(s + t - x\cdot\omega)
    \end{equation}
and
    \begin{equation}\label{est_M11_3}
        D_{s}^{l_1}R_N(t,x,s,\omega) \arrowvert_{t < -s - \rho} = 0.
    \end{equation}
Therefore, the same steps we used to prove Lemma \ref{est_M10}, will give
\begin{multline}\label{est_M11_4}
    \| \partial_\sigma^{l_1}A_{kj}^{l_2}(\sigma',\cdot,\cdot)\|_{L^2(\Sigma_\sigma\times\mathbb{S}_\omega^{n-1})} + \| \partial_\sigma^{l_1}B_{kj}^{l_2}(\sigma',\cdot,\cdot)\|_{L^2(\Sigma_\sigma\times\mathbb{S}_\omega^{n-1})}\\
    + \| \partial_\sigma^{l_1}A_{k}^{l_2}(\sigma',\cdot,\cdot)\|_{L^2(\Sigma_\sigma\times\mathbb{S}_\omega^{n-1})} + \| \partial_\sigma^{l_1}B_{j}^{l_2}(\sigma',\cdot,\cdot)\|_{L^2(\Sigma_\sigma\times\mathbb{S}_\omega^{n-1})} \leq \varepsilon C_{\Omega} \|q\|_{L^\infty(\mathbb{R}; L^2(\mathbb{R}^n))}
\end{multline}
for any $\sigma'\in \mathbb{R}$. 

Next, we estimate
\begin{multline*}
    |B_{k}^{l_1l_2}(\sigma',\sigma,\omega)| \leq \|a_{2,k}\|_{L^{\infty}(\mathbb{R}\times\Omega\times\mathbb{S}^{n-1})} \\
    \times\int_{\Omega} \|q(\cdot, x)\|_{L^\infty(\mathbb{R})} \int_{-T}^T |\partial_{\sigma}^{l_1}R_{1,N}(t-\sigma',x,\sigma' + \sigma,\omega)| h_k^{(l_2)}(t - \sigma + x\cdot\omega) dtdx
\end{multline*}
Due to Lemma \ref{est_for_aR}, it follows 
\begin{multline*}
    \|B_{k}^{l_1l_2}(\sigma',\cdot,\cdot)\|_{L^\infty(\Sigma\times\mathbb{S}^{n-1})}\\
    \leq  C_\Omega \|q\|_{L^\infty(\mathbb{R};L^2(\Omega))} \sup_{\sigma\in \Sigma}\sup_{\omega\in\mathbb{S}^{n-1}} \sup_{x\in \Omega}\sup_{t\in [-T,T]} |\partial_{\sigma}^{l_1}R_{1,N}(t-\sigma',x,\sigma' + \sigma,\omega)|.
\end{multline*}
To estimate the right-hand side, we repeat some steps of Lemma \ref{est_for_aR}. Since $\partial_s^{l_1}R_{1,N}$ satisfies \eqref{est_M11_2} and \eqref{est_M11_3}, Lemma \ref{energ_est} gives 
\begin{multline*}
    \sup_{\sigma\in \Sigma} \sup_{t\in [-T,T]} |\partial_{\sigma}^{l_1}R_{1,N}(t-\sigma',x,\sigma' + \sigma,\omega)|\\ \leq C_\Omega \int_{-\sigma-\sigma' - \rho}^{T-\sigma'} \|A_{1,N}(\tau,\cdot,\omega)h_N^{l_1}(T+\sigma - (\cdot)\cdot \omega\|_{H^K(\mathbb{R}^n)}d\tau,
\end{multline*}
where $K = (n-1)/2 + 1$ and $A_{1,N}$ is the function defined by \eqref{A1N}. Let
\begin{equation*}
    z=\sup_{\sigma\in\Sigma}T+\sigma.
\end{equation*}
Due to \eqref{est_A1N_hN},
\begin{equation*}
    \sup_{\sigma\in \Sigma} \sup_{t\in [-T,T]} |\partial_{\sigma}^{l_1}R_{1,N}(t-\sigma',x,\sigma' + \sigma,\omega)| \leq C_{\Omega} \sup_{\tau\in\mathbb{R}}\|A_{1,N}(\tau,\cdot,\omega)\|_{H^K(\tilde{\Omega})}
\end{equation*}
for some compact $\tilde\Omega$ which depends only on $\Omega$. Then, \eqref{est_A1N} gives
\begin{equation*}
    \sup_{\sigma\in \Sigma} \sup_{t\in [-T,T]} |\partial_{\sigma}^{l_1}R_{1,N}(t-\sigma',x,\sigma' + \sigma,\omega)| \leq C_{\Omega} \|q_1\|_{C^{L + \frac{n-1}{2}+3+2N}(\mathbb{R}\times\bar\Omega)} \leq \varepsilon C_\Omega,
\end{equation*}
and hence,
\begin{equation*}
    \|B_{k}^{l_1l_2}(\sigma',\cdot,\cdot)\|_{L^\infty(\Sigma\times\mathbb{S}^{n-1})} \leq \varepsilon C_\Omega \|q\|_{L^\infty(\mathbb{R};L^2(\Omega))}.
\end{equation*}
Similarly, this estimate holds also for $A_{j}^{l_1l_2}$ and $C^{l_1l_2}$. Hence, from \eqref{est_M11_5} and \eqref{est_M11_4}, we obtain \eqref{est_M11_1}.
\end{proof}

Finally, we prove Theorem \ref{thm_1}.
\begin{proof}[Proof of Theorem \ref{thm_1}]
    Let $C$ be a constant form Lemma \ref{est_M00} and $C_\Omega$ be a common constant from Lemmas \ref{est_M10} and \ref{est_M11}. Let us fix $0<\varepsilon<1$ such that $1/C - 3\varepsilon C_\Omega>0$. Next, we choose $k\in \mathbb{N}$ as large as Lemmas \ref{est_M10} and \ref{est_M11} require. Then, under conditions $\|q_1\|_{C^k(\mathbb{R}\times \bar\Omega)}$, $\|q_1\|_{C^k(\mathbb{R}\times \bar\Omega)}<\varepsilon$, Lemmas \ref{est_M00}, \ref{est_M10}, and \ref{est_M11} imply
    \begin{equation*}
        \|\tilde{M}q\|_{L^\infty(\mathbb{R}_{\sigma'}; L^2(\mathbb{S}_\omega^{n-1}; H^{(n-1)/2}(\Sigma_\sigma)))} \geq \frac{1}{C}\|q\|_{L^\infty(\mathbb{R}; L^2(\mathbb{R}^n))} - 3 \varepsilon C_{\Omega} \|q\|_{L^\infty(\mathbb{R}; L^2(\mathbb{R}^n))}.
    \end{equation*}
    This completes the proof.
\end{proof}

\begin{remark}
We want to emphasize on some subtle moment in the proof.  The integration in \r{pr1-est} happens inside the wedge in Figure~\ref{time-dep-pot_pics1}, which, intersected with the cylinder $|x|\le \rho$, is compact. On the other hand, $M_{00}q$ integrates over the ``line segment'' (a hyperplane) there only while the other integrals integrate  $q$ inside the whole wedge. In order to absorb the $\tilde M_{10}q$, etc., terms, we need them to be small in $q$, which they are but they depend on $q$ over a set larger than the one needed in $\tilde M_{00}q$. This arguments still works because we actually extend the estimates to $q$ everywhere in the $t$ variable by taking a supremum in $\sigma'$ above. On the other hand, if we wanted to establish local stability, like estimating $q$ for $t$ over a finite time interval having finite time back-scattering data, that would have been be a problem. 
\end{remark}

\appendix
\section{Scattering theory for time-dependent potentials} \label{sec_appendix} 
We recall the basics of the scattering theory for time-dependent perturbations of the wave equation by restricting it to time-dependent potentials. We follow Cooper and Strauss \cite{CooperS,CooperS84}, where it was introduced for moving obstacles, and its adaptation to time-dependent potentials in \cite{MR1004174}. Some of the statements below are new however, like Theorem~\ref{thm_TD_AWP2} and Theorem~\ref{TD_thm_sc1}. This theory is a natural extension of (a part of) the Lax-Phillips scattering theory. We consider  the wave equation \r{TD_14} with  a smooth time dependent potential $q(t,x)$ supported in the cylinder $\R\times\overline{B(0,R)}$.  We assume $n\ge3$, odd to avoid working with the non-local translation representation when $n$ is even. 

\subsection{Lax-Phillips formalism about the wave equation} 
The natural Cauchy problem for the wave equation is the following
\be{TD_5}
(\partial_t^2-\Delta)u=0, \quad (u,u_t)|_{t=0} =  (f_1,f_2).
\ee
 We convert the wave equation into a system by setting $\mathbf{u}(t)=(u,u_t)$; then
\be{TD_A}
\partial_t \mathbf{u} = A\mathbf{u}, \quad A:= \begin{pmatrix} 
   0 & \Id \\
   \Delta & 0
  \end{pmatrix}.
\ee
We use boldface to denote vector-valued functions not necessarily of the type $(u,u_t)$ if there is no background scalar function $u(t,x)$ present. In particular, $\mathbf{u}(t)$ in Definition~\ref{def_out} below is not necessarily of that form. 

The natural energy space of states of finite energy is defined as the completion of $C_0^\infty\times C_0^\infty$ under the energy norm
\[
\|\mathbf{f}\|^2_{\mathcal{H}}= \frac12\int \left( |\nabla f_1|^2+|f_2|^2     \right)\d x, \quad \mathbf{f}:=(f_1,f_2). 
\]
In particular, the first term defines the Dirichlet space $H_D(\R^n)$ with norm $\|\nabla f\|_{L^2}$.  When $n\ge3$, they are locally in $L^2$, as it follows from the  Poincar\'e inequality. 
The operator $A$ naturally extends to a skew-selfadjoint one (i.e, $\i A$ is self-adjoint) on $\mathcal{H}$. Then by Stone's theorem, $U_0(t) = e^{tA}$ is a well-defined strongly continuous unitary group, and the solution of \r{TD_A} is given by $\mathbf{u}(t) = U_0(t)\mathbf{f}$. The unitarity means energy conservation, in particular. 

We define the local energy space $\mathcal{H}_\textrm{loc}$ in the usual way. By the finite speed of propagation, the Cauchy problem \r{TD_5} has a well defined solution in $\mathcal{H}_\textrm{loc}$ if the Cauchy data $\mathbf{f}$ is in  $\mathcal{H}_\textrm{loc}$ only. We view those solutions as ones with (possibly) infinite energy but locally finite one.
 Then $\mathbf{u}\in C(\R;\; \mathcal{H}_\textrm{loc})$ and the wave equation is solved in distribution sense. 
 One can easily extend this to distributions. 
 
 \subsection{Existence of dynamics} 
By \cite{Kato_70}, see also \cite[X.12]{Reed-Simon2}, the solution to 
\be{TD_17}
(\partial_t^2-\Delta+q(t,x))u=0,\quad    (u,u_t)|_{t=s}=(f_1,f_2)
\ee
is given by $\mathbf{u}(t)= U(t,s)\mathbf{f}$, where $\mathbf{f}=(f_1,f_2)$ and $U(t,s)$ is a two-parameter strongly continuous group of bounded operators with the properties
 \begin{itemize}
   \item[(i)] $U(t,s)U(s,r) = U(t,r)$ for all $t,s,r$; and $U(t,t)=\Id$, 
   \item[(ii)] $\|U(t,s)\|\le \exp\left\{ C|t-s|\sup_{s\le\tau\le t, \; x\in \R^n}|q(\tau,x)|  \right\}$,
   \item[(iii)] for any $\mathbf{f}\in D(A)$, we have $U(t,s)\mathbf{f}\in D(A)$ and
\be{TD_UU}
\frac{\d}{\d t} U(t,s)\mathbf{f} = (A-Q(t))U(t,s)\mathbf{f}, \quad 
\frac{\d}{\d s} U(t,s)\mathbf{f} = -U(t,s)(A-Q(s))\mathbf{f},
\ee 
\end{itemize}
where $Q(t)\mathbf{f}= (0,q(t,\cdot)f_1)$ (and $Q(t)$ is clearly bounded). 

The two-parameter semi-group admits the expansion
\be{TD_exp}
U(t,s) = U_0(t-s)+\sum_{k=1}^\infty V_k(t,s),
\ee
where
\[\begin{split}
V_k(t,s)\mathbf{f}& = (-1)^k \int_s^t \d s_1 \int_s^{s_1} \d s_k \dots \int_s^{s_{k-1}}\d s_k\\
&\qquad \times  U_0(t-s_1)Q(s_1) \dots U_0(s_{k-1}-s_k) Q(s_k)U_0(s_k-s)\mathbf{f}, \quad k\gg1.
\end{split} 
\]
This expansion is an iterated version of the Duhamel's formula
\be{TD_Duh}
\begin{split}
U(t,s) &= U_0(t-s)+ \int_s^t U(t,\sigma)Q(\sigma)U_0(\sigma-s) \,\d \sigma\\
&= U_0(t-s)+ \int_s^t U_0(t-\sigma)Q(\sigma)U(\sigma,s)\,\d \sigma. 
\end{split}
\ee
The convergence of \r{TD_exp} follows from the estimate
\[
\|V_k(t,s)\|\le \frac{|t-s|^k}{k!} \left( \sup_{s\le\tau\le t}\|Q(\tau)\| \right)^k.
\]
In particular, we get that we still have the finite speed of propagation property:
\[
\supp U(t,s)\mathbf{f}\subset \supp\mathbf{f}+B(0,|t-s|). 
\]
As before, the finite speed of propagation allows us to can extend $U(t,s)$ to the space $\mathcal{H}_\textrm{loc}$ by a partition of unity. 

Finally, notice that when $q$ is time independent, then $U(t,s)$ depends on the difference $t-s$ only, i.e., $U(t,s)=U(t-s)$ where $U$ is a group. It is not unitary however (unless $q=0$) in the space $\mathcal{H}$. If we redefine the energy norm by
\be{TD_q}
\|\mathbf{f}\|^2_{\mathcal{H}^q}= \int \left( |\nabla f_1|^2+q|f_1|^2+ |f_2|^2     \right)\d x, 
\ee
(we need to know that it is a norm however, and $q\ge0$ suffices for that), then $U(t)$ is unitary in $\mathcal{H}^q$. 

\subsection{Plane waves, translation representation  and asymptotic wave profiles of free solutions}  The   plane waves
\[
\delta(t-\omega\cdot x)
\]
solve the wave equation, obviously. They can be thought of as plane waves propagating in the direction $\omega$ with speed one. If we replace $t$ by $t+s$ there, we can think of $s$ as the delay time.  The  plane wave above is the Schwartz kernel of the Radon transform
\[
Rf(s,\omega) = \int \delta(s-\omega\cdot x)f(x)\,\d x= \int_{x\cdot\omega=s}f(x)\,\d S_x.
\]
For any density $g(\omega,s)$ (which can be a distribution as well), the superposition
\be{TD_7a}
u(t,x) := \int_{\R\times S^{n-1}}\delta(t+s-\omega\cdot x)g(s,\omega)\,\d s\,\d \omega  = \int_{S^{n-1}} g(\omega\cdot x-t,\omega)\,\d\omega
\ee
is still a solution of the wave equation. The expression above can be recognized as the the transpose $R'$ of the Radon transform applied to $g_t(s,\omega):= g(s-t,\omega)$. 
It turns out that all solutions of the free wave equation in the energy space have that form. 

Indeed, in \cite{LP}, Lax and Phillips defined the  \textit{free translation representation} $\mathcal{R}: \mathcal{H}\to L^2(\R\times S^{n-1})$ as follows
\be{TD_8}
k(s,\omega) = \mathcal{R}\mathbf{f}(s,\omega) = c_n(-\partial_s^{(n+1)/2} Rf_1+\partial_s^{(n-1)/2} Rf_2),
\ee
where $R$ is the Radon transform and $c_n=2^{-1}(2\pi)^{(1-n)/2}$, $c_n^-= 2^{-1}(-2\pi)^{(1-n)/2}$. The inverse is given by
\be{TD_9}
\mathcal{R}^{-1}k(x) = 2c_n^-\int_{S^{n-1}}\left(-\partial_s^{(n-3)/2} k(x\cdot\omega,\omega), \; \partial_s^{(n-1)/2}k(x\cdot\omega,\omega)\right)\d\omega. 
\ee
The map $\mathcal{R}$ is unitary, and $(\mathcal{R}U_0(t)\mathcal{R}^{-1} k)(s,\omega) = k(s-t,\omega)$, which explains the name. We also set
\be{TD_10}
u^\sharp (s,\omega) = (-1)^{(n-1)/2} k(s,\omega)
\ee
and call $u^\sharp$ the \textit{asymptotic wave profile} of the solution $\mathbf{u}(t)= U_0(t)\mathbf{f}$. This name is justified by the theorem below, and it is the analog of the far free pattern for solutions of the free wave equation.

\begin{theorem}[Lax-Phillips, \cite{LP}] 
Let $\mathbf{u}(t)= U_0(t)\mathbf{f}$, $\mathbf{f}\in \mathcal{H}$. Then
\be{TD_11}
\int\Big| u_t-|x|^{- (n-1)/2 }u^\sharp\Big(|x|-t,\frac{x}{|x|}\Big) \Big|^2\, \d x \to 0, \quad \text{as $|t|\to\infty$}. 
\ee
\end{theorem}

\begin{remark}
In \cite{LP}, the factor $(-1)^{(n-1)/2}$ is missing from \r{TD_10}, i.e., $u^\sharp= k$. Cooper and Strauss in \cite{CooperS} found out that this factor must be present in \r{TD_10}. 
\end{remark}

\subsection{Outgoing solutions and their asymptotic wave profiles} \label{sec_OS}
We follow here \cite{Cooper_Strauss_79, CooperS}. 
 Given $u(t,x)$ (and only then), recall the notation $\mathbf{u}(t) :=(u(t,\cdot),u_t(t,\cdot))$, see \r{TD_A}.

\begin{definition}\label{def_out}
The function  $\mathbf{u}(t)\in C(\R;\;\mathcal{H}_\textrm{\rm loc})$ is called outgoing if $\lim_{t\to-\infty} (\mathbf{u}(t),U_0(t) \mathbf{g})=0$ for each $\mathbf{g}\in C_0^\infty(\R^n) \times C_0^\infty(\R^n)$. 
\end{definition}

In this definition, $\mathbf{u}(t)$ does not need to be a solution of the wave equation (anywhere). On the other hand, if $u(t,x)$ solves the wave equation in $|x|>\rho$ for some $\rho>0$, then, see \cite{Cooper_Strauss_79, CooperS}, $u$ is outgoing if and only if for any $T\in \R$, $U_0(t-T)\mathbf{u}(T)=0$ in the forward cone $|x|<t-T-\rho$. 

One simple example of non-trivial outgoing solutions (for $|x|>\rho)$ is the following. Let $p\in L^1_\textrm{loc}(\R;\;L^2(\R^n))$ satisfy $p=0$ for $t<t_0$, where $t_0$ is fixed. 
Solve
\be{TD_13}
(\partial_t^2-\Delta)u=p(t,x) \quad \text{in $\R\times\R^n$}. 
\ee
 with Cauchy data 
\[
(u,u_t)|_{t=t_0}=(0,0).
\]
 By  Duhamel's formula,
\[
\mathbf{u}(t)= \int_{t_0}^t U_0(t-s)\mathbf{p}(s)\,\d s, \quad \mathbf{p}(s) := (0,p(s,\cdot)).
\]
The latter is well-defined in $\mathcal{H}_\textrm{loc}$ by finite speed of propagation. The solution for $t<t_0$ is just zero. Then $\mathbf{u}$ is outgoing in a trivial way. Moreover, this is the unique outgoing solution of \r{TD_13}. Indeed, take the difference $v$ of any two. Then $\mathbf{v}(t)=U_0(t)\mathbf{f}$, where $\mathbf{f}$ is the initial condition. Then $0=\lim_{t\to-\infty} (\mathbf{v}(t),U_0(t) \mathbf{g})=  (\mathbf{f},\mathbf{g})$, for any test function $\mathbf{g}$; therefore, $\mathbf{f}=0$ and then $\mathbf{v}=0$.

This can be generalized as follows.
\begin{theorem}[\cite{Cooper_Strauss_79, CooperS}] \label{TD_thm_3}
Let $p\in L^1_\textrm{\rm loc}(\R;\;L^2(\R^n))$ and assume that for each $t$, 
\be{TD_12a}
 \lim_{T\to-\infty}\int_{T}^t  U_0(-s)\mathbf{p}(s)\,  \d s  \quad\text{exists in $\mathcal{H}_\textrm{\rm loc}$},  \quad \mathbf{p}(s) := (0,p(s,\cdot).
\ee
 Then there exists a unique outgoing solution $\mathbf{u}\in C(\R;\;\mathcal{H}_\textrm{loc})$ of \r{TD_13} given by
\[
\mathbf{u}(t)= \int_{-\infty }^t U_0(t-s)\mathbf{p}(s)\,\d s.
\]
\end{theorem} 
\begin{remark}
Clearly, $p\in L^1((-\infty,a);\;L^2(\R^n))$ for any $a$ would guarantee the regularity assumption on $p$ and \r{TD_12a}. Also, the assumptions on $p$ in next theorem are enough. 
\end{remark} 
\begin{proof}
The absolute convergence of the integral in ${H}^0_\textrm{loc}$ follows from the assumptions. To show  that $u$ is outgoing,  for $\mathbf{g} \in C_0^\infty\times C_0^\infty$, consider
\[
(\mathbf{u}(t), U_0(t)\mathbf{g})= \int_{-\infty}^t\left(  U_0(t-s)\mathbf{p}(s), U_0(t)\mathbf{g} \right) \d s = \int_{-\infty}^t\left(  U_0(-s)\mathbf{p}(s), \mathbf{g} \right) \d s.
\]
The latter converges to $0$, as $t\to -\infty$ by assumption. 
\end{proof}

\begin{theorem}[\cite{Cooper_Strauss_79, CooperS}]  \label{thm_TD_AWP} 
Let $n\ge3$ be odd. Let  $p\in L^1_\textrm{\rm loc}(\R;\;L^2(\R^n))$ with $p(t,x)=0$ for $|x|>\rho$. Let $u$ be the unique outgoing solution of  \r{TD_13}. 

(a) Then there is a unique function $u^\sharp \in L_\textrm{\rm loc}^2(\R\times S^{n-1})$ such that for all $R_1<R_2$ we have
\[
\int_{R_1+t<|x|<R_2+t} \left| u_t(t,x) - |x|^{-(n-1)/2} u^\sharp \left( |x|-t,\frac{x}{|x|}\right) \right|^2\d x\to0, \quad \text{\rm as $t\to\infty$}. 
\]

(b) If $p\in C_0^\infty$, 
\be{T_D_13a}
u^\sharp(s,\omega) = c_n^- \partial_s^{(n-1)/2}\int p(\omega\cdot x-s,x)\,\d x.
\ee

(c) The map $p\to u^\sharp$ is continuous.
\end{theorem}

\begin{remark}
For general $p$ as in the theorem, $u^\sharp$ is still given by \r{T_D_13a} but the derivative is in distribution sense; by (b), the result is in  $L_\textrm{\rm loc}^2(\R\times S^{n-1})$. Another way to write  \r{T_D_13a} is 
\be{T_D_14a}
\int u^\sharp (s,\omega)\phi(s)\,\d s = c_n \iint p(t,x) (\partial_s^{(n-1)/2} \phi) (\omega\cdot x-s,x)\,\d t\, \d s, \quad \forall \psi(s)\in C_0^\infty(\R). 
\ee 
\end{remark}

\begin{proof}
Motivated by Theorem~\ref{TD_thm_3}, for fixed $R_1<R_2$, set
\[
\mathbf{f} = \int_{-R_2-\rho}^{-R_1+\rho}U_0(-\tau)\mathbf{p}(\tau)\,\d \tau, \quad \mathbf{v}(t) = U_0(t)\mathbf{f}. 
\]
By Huygens' principle, $\mathbf{v}(t)=\mathbf{u}(t)$ for $R_1+t< |x|<R_2+t$.  Therefore, $\mathbf{v}$ does have an asymptotic wave profile, and    $v^\sharp(s,\omega) = u^\sharp(s,\omega)$ for $R_1<s<R_2$. On the other hand, we have a formula for $v^\sharp$, \r{TD_9} and \r{TD_10} which says
\[
\begin{split}
v^\sharp(s,\omega) &= (-1)^{(n-1)/2}(\mathcal{R}\mathbf{f})(s,\omega) \\
 &= c_n^-\partial_s^{(n-1)/2} \int_{-R_2-\rho}^{-R_1+\rho} \int_{x\cdot\omega=s+\tau} p(x\cdot\omega-s,x)\, \d S_x\, \d\tau.
\end{split}
\]
Then
\[
u^\sharp(s,\omega)|_{R_1<s<R_2} = v^\sharp(s,\omega)= c_n^- \partial_s^{(n-1)/2}\int p(x\cdot\omega-s,x)\, \d x.
\]
Since $R_1<R_2$ are arbitrary, this, combined with Theorem~\ref{TD_thm_3}, proves (a);  and (b) for $p\in C_0^\infty$.

The proof of (c) is straightforward: use \r{T_D_13a} and take Fourier transform w.r.t.\ $s$. In particular, we get that the map $p\to u^\sharp$ can be extended continuously in those spaces. 
\end{proof} 

\begin{remark}
We call $u^\sharp$ in the theorem the asymptotic wave profile of the unique outgoing solution of \r{TD_13}. Note that there are two cases where we defined such profiles: for free solutions in the energy space, see Theorem~\ref{TD_thm_3}, and in Theorem~\ref{thm_TD_AWP} above, where $u$ is in the energy space locally only, and solves \r{TD_13} instead.  
\end{remark}

\subsection{Scattering solutions} 
The scattering solutions $u^-$ and $u^+$ were introduced in section~\ref{sec_3} as the solutions of  \r{TD_14}, and \r{TD_14in}, respectively. Since they involve distributions, not necessarily in the energy spaces (even locally), we proceed as follows. We can think of $(u(t,x;s,\omega), u_t(t,x;s,\omega))$ as distribution in the $(s,\omega)$ variables with values in $\mathcal{H}_\textrm{loc}$. It is more convenient however to do the following. Let $h_j(t)=h(t)t^j/j!$, $j=1,2\dots$, where $h$ is the Heaviside function; and we also set $h_{-1}=\delta$ . Then $h_j'=h_{j-1}$, $j=0,1,2,\dots$. To define $u^-$ eventually, we  solve
\be{TD_15}
(\partial_t^2-\Delta+q(t,x))\Gamma=0,   \quad \Gamma|_{t<-s-\rho }= h_1(t+s-x\cdot\omega)
\ee
first (notice that $  h_1(t+s-x\cdot\omega) $ is locally in the energy space now), set
\be{TD_16}
\Gamma_\textrm{sc} = \Gamma - h_1(t+s-x\cdot\omega),
\ee
compute the asymptotic wave profile $\Gamma^\sharp(s',\omega';s,\omega)$ of $\Gamma_\textrm{sc}$, and differentiate the result twice w.r.t.\ $s$ to get the  analog of the scattering amplitude. In particular, then
\be{TD_16'}
u(t,x;s,\omega) = \partial_s^2\Gamma(t,x;s,\omega), \quad u_\textrm{sc}^-(t,x;s,\omega) = \partial_s^2\Gamma_\textrm{sc}(t,x;s,\omega).
\ee
will be well defined as distributions. 

In a similar way, one can construct the scattering solutions $u^+$ which look like plane waves as $t\to+\infty$, instead of $t\to-\infty$.   They would solve \r{TD_14in}. 
Performing the change of variables $\tilde t=-t$ (time reversal), $\tilde s=-s$, $\tilde \omega=-\omega$, we see that $u^+(t,x;s,\omega) =\tilde u^+(-t,x;-s,-\omega) $, where $\tilde u$ is related to $\tilde q(t,x):=q(-t,x)$. The regularized version, $\Gamma^+$, can be constructed as in \r{TD_15} with the condition $\Gamma^+ = h_1(-t-s+x\cdot\omega)$ for $t>-s+\rho$. The right-hand side of this condition then would be supported outside $\R\times B(0,R)$ for $t>-s+\rho$. Then we define $u^+$ and $u^+_\textrm{sc}$ as in \r{TD_16'}.

By the finite speed of propagation property 
\be{FS}
\supp u^-(\cdot,\cdot;s,\omega) \subset \{t+s-x\cdot\omega\ge0  \}, \quad \supp u^+(\cdot,\cdot;s,\omega) \subset \{t+s-x\cdot\omega\le0  \}.
\ee

Next theorem generalizes Theorem~\ref{thm_TD_AWP}.

\begin{theorem} \label{thm_TD_AWP2} 
Let $p$ be as in Theorem~\ref{thm_TD_AWP}. Set
\[
\mathbf{u}(t): = \int_{-\infty}^t U(t,s) \mathbf{p}(s)\,\d s.
\]
Then $\mathbf{u}\in C(\R;\;\mathcal{H}_\textrm{loc})$ is outgoing, and has 
has an asymptotic wave profile $u^\sharp(s,\omega)$ given by  
\be{A1}
u^\sharp(s,\omega) = c_n^- \partial_s^{(n-1)/2}\int p(t,x) u^+ (t,x;s,\omega)\,\d t\,\d x.
\ee
\end{theorem}

\begin{proof}
By \r{TD_Duh},
\be{A2}
\begin{split}
\mathbf{u}(t) &= \int_{-\infty}^t U_0(t-s) \mathbf{p}(s)\,\d s + 
\int_{-\infty}^t\int_s^t  U_0(t-\sigma)Q(\sigma) U(\sigma,s)  \mathbf{p}(s)\,\d \sigma \,\d s\\
& = \int_{-\infty}^t U_0(t-s) \mathbf{p}(s)\,\d s + 
\int_{-\infty}^t\int_{-\infty}^\sigma U_0(t-\sigma)Q(\sigma) U(\sigma,s)  \mathbf{p}(s)\,\d s\,\d \sigma .
\end{split}
\ee
Then we are in the situation of  Theorem~\ref{thm_TD_AWP}  with $\mathbf{p}(t)$ there replaced by 
\be{A2a}
\tilde{\mathbf{p}}(t):= \mathbf{p}(t)+ \mathbf{p}_1(t), \quad \mathbf{p}_1(t):= Q(t)\int_{-\infty}^t   U(t,s)  \mathbf{p}(s)\,\d s=  Q(t)\mathbf{u}(t). 
\ee
Then $\mathbf{u}$ has an asymptotic wave profile $u^\sharp(s',\omega')$ satisfying 
\be{T_D_13b}
u^\sharp(s',\omega') = c_n^- \partial_s^{(n-1)/2}\int\tilde p(t,x)\delta(t+s'-\omega'\cdot x)\,\d x\,\d t.
\ee
The first term on the right-hand side of \r{A2} is handled by Theorem~\ref{thm_TD_AWP}. We analyze the second term below, which we call $\mathbf{u}_1(t)$. By Theorem~\ref{thm_TD_AWP} again, its asymptotic wave profile is given by 
\be{A2'}
\begin{split}
u_1^\sharp(s',\omega')&= c_n^- \partial_s^{(n-1)/2}\int \Big[ Q(t,x)\int_{-\infty}^t   U(t,x;s,y)  \mathbf{p}(s,y)\,\d s\,\d y\Big]_2\,  {\delta}(t+s'-\omega'\cdot x)\,\d t\,\d x\\
&= c_n^- \partial_s^{(n-1)/2} \int K(s',\omega';s,y) p(s,y)\,\d s\,\d y,
\end{split}
\ee
where the last identity defines $K$, i.e., 
\be{A3}
K(s',\omega';s,y)= \int^{\infty}_s \int q(t,x)   U_{12}(t,x;s,y)   \delta(t+s'-\omega'\cdot x)\,\d x\,\d t.
\ee
By \r{TD_UU}, 
\[
(-\partial_s + A_y'- Q'(s)) U'(t,x;s,y)=0,
\]
where the primes denote  transpose operators  in distribution (not in energy space) sense. This equality can be written also as 
\[
 \begin{pmatrix} 
   -\partial_s  & \Delta_y-q(s)\\ \Id & -\partial_s 
  \end{pmatrix}U'(t,x;s,y) =0 .   
\]
In particular,
\be{A4}
(\partial_s^2-\Delta_y+q(s)) U_{12}(t,x;s,y)=0.
\ee
Differentiate $K$ in \r{A3} to obtain
\[
\partial_s K(s',\omega';s,y)= \int^{\infty}_s \int q(t,x) \partial_s  U_{12}(t,x;s,y)   \delta(t+s'-\omega'\cdot x)\,\d x\,\d t
\]
because  $U_{12}(s,x;s,y)=0$. Differentiate again:
\[
\begin{split}
\partial_s^2 K(s',\omega';s,y)&= \int^{\infty}_s \int q(t,x) \partial_s^2  U_{12}(t,x;s,y)   \delta(t+s'-\omega'\cdot x)\,\d x\,\d t \\
&\qquad -q(s,x)\delta(s+s'-\omega'\cdot y).
\end{split}
\]
Then by \r{A4} and \r{A3},
\be{A5}
(\partial_s^2-\Delta_y+q(s)) K(s',\omega';s,y)= -q(s,x)\delta(s+s'-\omega'\cdot y).
\ee
On the support of the integrand in \r{A3}, we have $t+s'<\rho$, $s<t$. Therefore, 
\be{A6}
 K(s',\omega';s,y)|_{s>-s'+\rho}=0.
\ee
Therefore, $K$ solves \r{A5}, \r{A6}, which  is the same problem solved by $u^+_\textrm{sc}(s',\omega';s,y)$, see \r{TD_14in}. Therefore, $K=u^+_\textrm{sc}$.

Going back to \r{T_D_13b} and \r{A2a}, we see that
\be{A6'}
\begin{split}
u^\sharp(s',\omega') &= c_n^- \partial_s^{(n-1)/2}\int\left(  p(t,x)+p_1(t,x) \right)\delta(t+s'-\omega'\cdot x)\,\d x\,\d t\\
& =  c_n^- \partial_s^{(n-1)/2}\int   p(t,x) )
\left( \delta(t+s'-\omega'\cdot x)+ u^+_\textrm{sc}(s',\omega';t,x)\right) \,\d x\,\d t\\
& =  c_n^- \partial_s^{(n-1)/2}\int   p(t,x) )
  u^+(s',\omega';t,x) \,\d x\,\d t,
\end{split}
\ee
where we used \r{A2'} and the identity $K=u^+_\textrm{sc}$ we just derived. 
\end{proof}

\subsection{The scattering amplitude and the scattering kernel} 
 Let $\Gamma$ solve \r{TD_15}. Since the Cauchy data $(h_1(t+s-x\cdot\omega), h_0(t+s-x\cdot\omega))$, for  say, $t=-s-\rho-1$, is in $\mathcal{H}_\textrm{loc}$, a solution $(\Gamma,\Gamma_t)$ with locally finite energy exists. Then $\Gamma_\textrm{sc}$ is clearly outgoing. It solves the Cauchy problem
\be{TD_18}
(\partial_t^2-\Delta )\Gamma_\textrm{sc}= -q\Gamma,\quad    \Gamma_\textrm{sc}|_{t<-s-\rho}= 0. 
\ee
By Theorem~\ref{thm_TD_AWP}, $\Gamma_\textrm{sc}$ has an asymptotic wave profile $\Gamma_\textrm{sc}^\sharp $ given by 
\[
\begin{split}
\Gamma_\textrm{sc}^\sharp (s',\omega';s,\omega) &= -c_n^-\partial_{s'}^{(n-1)/2} \int q (x\cdot\omega'-s' ,x )\Gamma(x\cdot\omega'-s' ,x;s,\omega)\, \d x\\
  &=  -c_n^-\partial_{s'}^{(n-1)/2} \int q (t ,x )\Gamma(t ,x;s,\omega)\delta(t+s'- x\cdot\omega' )\, \d t\,\d x.
\end{split}
\]
Differentiate twice w.r.t.\ $s$, see  \r{TD_16'}, to get  
\[
u^{-,\sharp}_\textrm{sc} (s',\omega';s,\omega)= -c_n^-\partial_{s'}^{(n-1)/2} \int q (t ,x )u^-(t ,x;s,\omega)\delta(t+s'- x\cdot\omega' )\, \d t\,\d x.
\]
\begin{definition} \label{def_sc_a}
The scattering amplitude $A^\sharp$ is given by
\be{TD_18_def}
A^\sharp(s',\omega';s,\omega) = \int q (t ,x )u^-(t ,x;s,\omega)\delta(t+s'- x\cdot\omega' )\, \d t\,\d x,
\ee
where $u^-$ solves \r{TD_14}. 
\end{definition}

By the finite speed of propagation, $u^-(t,x;s,\omega) = 0$ for $x\cdot\omega>t+s$. Therefore, the integrand vanishes outside of the region $x\cdot(\omega-\omega')\le s-s'$.  The l.h.s.\ has a lower bound $-2\rho$ on $\supp q$; therefore,
\be{TD_19}
\supp A^\sharp \subset \{ s'\le s+\rho|\omega-\omega'|\}\subset \{ s'\le s+2\rho\}.
\ee
Note that $A^\sharp$ and $u^\sharp_\textrm{sc} = -c_n^-\partial_{s'}^{(n-1)/2} A^\sharp$ can be reconstructed  from each other thanks to that support property.

Since the perturbed dynamics is a two-parameter group, we need to generalize the notion of the wave operators and the scattering scattering operator.

\begin{definition}\label{TD_def_SO}
The wave operators $\Omega_-$ and $W_+$ in $\mathcal{H}$ are defined as the strong limits
\[
\Omega_-  = \textrm{s\,-}\lim_{t\to-\infty}U(0,t)U_0(t), \quad
W_+  \mathbf{f}=  \lim_{t\to\infty}U_0(-t)U(t,0)\mathbf{f}; \quad \mathbf{f}\in \textrm{\rm Ran}\;\Omega_-,
\]
if they exist and define continuous  operators. In the latter case, the scattering operator $S$ is defined by
\[
S= W_+\Omega_-. 
\]
\end{definition}  
This definition also makes sense for  $\mathbf{f}\in \mathcal{H}_\textrm{comp}$ with $S\mathbf{f}$ taking values possibly in $\mathcal{H}_\textrm{loc}$.

\begin{theorem}\label{TD_thm_sc1}\  

(a) The wave operator $\Omega_-: \mathcal{H}_\textrm{\rm comp} \to \mathcal{H}$ exists 
and  
\be{TD_20}
U(t,0)\Omega_-\mathbf{f} =  2c_n^-\int_{\R\times S^{n-1}} \boldsymbol{u}^-(t,x;s,\omega) \partial_s^{(n-3)/2} (\mathcal{R}\boldsymbol{f})(s,\omega)\,\d s\,\d\omega. 
\ee

(b)  The wave operator $W_+:  \mathcal{H}\to \mathcal{H}_\textrm{\rm loc}$ exists. 

(c) The scattering operator $S:\mathcal{H}_\textrm{\rm comp}\to \mathcal{H}_\textrm{\rm loc} $ exists.
\end{theorem}

\begin{proof}
Choose $\mathbf{f}\in \mathcal{H}_\textrm{\rm comp}$, so that $\mathbf{f}(x)=0$ for $|x|>R$ with some $R>0$. Let $k=\mathcal{R}\mathbf{f}$. Then $k(s,\omega)=0$ for $|s|>R$. For $t<-R-\rho:=t_0$, $U(0,t)U_0(t)\mathbf{f}= U(0,t_0)U_0(t_0)\mathbf{f}$. In particular, the limit defining $\Omega_-\mathbf{f}$ exists trivially and $U(t,0)\Omega_-\mathbf{f} = U(t,t_0)U_0(t_0)\mathbf{f}$. The r.h.s.\ of the latter solves the perturbed wave equation and equals $U_0(t_0)\mathbf{f}= \mathcal{R}^{-1} k(\cdot-t_0,\cdot)$ for $t=t_0$. To prove \r{TD_20}, we need to show that the r.h.s.\ of \r{TD_20}, call it $\mathbf{v}(t)$, has the same initial condition  for $t\le t_0$. 

For $t\le t_0$, $u(t,x;s,\omega)= \delta(t+s-x\cdot\omega)$. Then by \r{TD_9}, 
\[
v(t) = 2c_n^-\int_{\R\times S^{n-1}} \delta(t+s-x\cdot\omega)\partial_s^{(n-3)/2}k(s,\omega)\,\d s\,\d\omega = (\mathcal{R}^{-1}k)_1(\cdot-t,\cdot), 
\]
which proves (a).

To prove the existence of $W_+$ in (b), fix first $R>0$ and let $\mathbf{1}_{B(0,R)}$ be the characteristic function of that ball. By \r{TD_Duh},
\be{TD_Duh'}
\mathbf{1}_{B(0,R)}U_0(-t) U(t,s) = \mathbf{1}_{B(0,R)}U_0(-s)+  \mathbf{1}_{B(0,R)} \int_s^t U_0(-\sigma)Q(\sigma)U(\sigma,s)\,\d \sigma.
\ee
By Huygens' principle, $\mathbf{1}_{B(0,R)} U_0(-\sigma)Q(\sigma)=0$ for $\sigma>R+\rho$. For $t>R+\rho$ then the integral above is independent of $t$ and therefore the strong limit $\mathbf{1}_{B(0,R)}W_+$ exists in a trivial way, defining a unique element in $\mathcal{H}_\textrm{loc}$.

Part (c) follows from (a) and (b). %
\end{proof}

The scattering operator $S$ on $\mathcal{H}$ exists (as a bounded operator) under some conditions, see the references in \cite[sec.~3]{MR1004174}. Then $-c_n^-\partial_{s'}^{(n-1)/2}  A^\sharp (s',\omega'; s,\omega) $ is the Schwartz kernel of $\mathcal{R}(S-\Id)\mathcal{R}^{-1}$. In the general case, we can consider the latter as the Schwartz kernel of the operator mapping asymptotic wave profiles instead of translation representations, see \r{TD_10}, as shown  \cite[Proposition~3.1]{MR1004174}. 

\begin{proposition}[\cite{Cooper_Strauss_79, CooperS}, \cite{MR1004174}]\label{pr_A}
Let $\mathbf{f}\in C_0^\infty\times C_0^\infty$. Let $\mathbf{v}_0(t)= U_0(t)\mathbf{f}$, and let $\mathbf{v}(t)$ be the solution of \r{eq1} which equals $\mathbf{v}_0(t)$ for $t\ll0$. Then we have
\[
v^\sharp(s',\omega') = v_0^\sharp(s',\omega')-c_n^-\partial_{s'}^{(n-1)/2}\int_{\R\times S^{n-1}} A^\sharp (s',\omega'; s,\omega) v_0^\sharp(s,\omega)\,\d s\,\d\omega.
\]
\end{proposition}

The proof of the proposition is done by taking the asymptotic wave profile of \r{TD_20} applying Duhamel's formula \r{TD_Duh} first.

\section{A weighted Radon transform}

We recall the definition of the weighted Euclidean Radon transform  
\begin{equation*}
    R_\mu f (p,\omega) = \int_{x\cdot \omega=p} \mu(x,\omega) f(x) dx = \int_{\omega^\perp} \mu (p\omega + y, \omega) f(p\omega + y) dy.
\end{equation*}
where $\mu\in C^\infty(\mathbb{R}\times \mathbb{S}^{n-1})$ is a weight. We will use the next result:
\begin{theorem}\label{w_Radon_t_est}
    Let $\Omega$ be an open, bounded set and $\mu\in C^\infty(\mathbb{R}\times \mathbb{S}^{n-1})$. Then, there is a constant $C_\Omega>0$, which depends only on $\Omega$, such that 
    \begin{equation*}
        \|R_\mu f\|_{L^2(\mathbb{S}_\omega^{n-1}; H^{(n-1)/2}(\mathbb{R}_\sigma))} \leq C_\Omega \sup_{\omega\in \mathbb{S}^{n-1}} \|\mu(\cdot,\omega)\|_{C^{2n - 2}(\bar\Omega)} \|f\|_{L^2(\Omega)}
    \end{equation*}
    for all $f\in C_0^\infty(\Omega)$.
\end{theorem}

To prove this, we need the following auxiliary lemmas.

\begin{lemma}\label{L2_norm_wF_trans}
    Let $\Omega\subset \mathbb{R}^n$ be an open, bounded set and $\nu$ be a function on $\mathbb{R}^n\times \mathbb{R}^n$ such that for any fixed $\xi\in \mathbb{R}^n$, $\nu(\cdot,\xi) \subset C^{n}(\bar\Omega)$ with $\mathrm{supp\,}\nu(\cdot,\xi) \subset \Omega$. Then, the operator $V$, given by
    \begin{equation*}
        V: f \rightarrow \int_{\mathbb{R}^n} e^{-ix\xi} \nu(x,\xi) f(x)dx,
    \end{equation*}
    satisfies
    \begin{equation*}
        \|V\|_{L^2(\Omega) \rightarrow L^2(\mathbb{R}^n)} \leq C_\Omega \sup_{\xi\in \mathbb{R}^n} \|\nu(\cdot,\xi)\|_{C^{n}(\bar\Omega)}.
    \end{equation*}
\end{lemma}
\begin{proof}
    Throughout this proof, $C_\Omega$ will serve as a universal positive constant depending only on $\Omega$, which may vary from line to line.  Let $f\in C_0^\infty(\Omega)$, then
    \begin{equation*}
       |Vf(\xi)| = \left|\int_{\mathbb{R}^n} D_{x_n}\cdots D_{x_1} \left(\int_{-\infty}^{x_1}\cdots\int_{-\infty}^{x_n}f(y)e^{-iy\xi}dy_{n} \cdots dy_{1} \right) \nu(x,\xi) dx\right|.
    \end{equation*}
    By integration by parts, we obtain
    \begin{multline*}
        |Vf(\xi)| \leq \int_{\mathbb{R}^n}  \left|\int_{-\infty}^{x_1}\cdots\int_{-\infty}^{x_n}f(y)e^{-iy\xi}dy_{n} \cdots dy_{1} \right| |D_{x_n}\cdots D_{x_1} \nu(x,\xi)| dx\\
        \leq C_\Omega \sup_{\xi\in \mathbb{R}^n} \|\nu(\cdot,\xi)\|_{C^{n}(\bar\Omega)} \sup_{x\in \mathbb{R}^n} \left|\int_{\mathbb{R}^n} e^{-iy\xi}f(y) \chi_{\{z: z_k\leq x_k\}}(y) dy\right|,
    \end{multline*}
    where $\chi_A$ is the indicator function for a set $A$. Therefore, we estimate
    \begin{align*}
        \|Vf\|_{L^2(\mathbb{R}^n)} &\leq C_\Omega \sup_{\xi\in \mathbb{R}^n} \|\nu(\cdot,\xi)\|_{C^{n}(\bar\Omega)} \sup_{x\in \mathbb{R}^n} \|\mathcal{F}[f\chi_{\{z: z_k\leq x_k\}}]\|_{L^2(\Omega)}\\
        & \leq C_\Omega \sup_{\xi\in \mathbb{R}^n} \|\nu(\cdot,\xi)\|_{C^{n}(\bar\Omega)} \sup_{x\in \mathbb{R}^n} \|f\chi_{\{z: z_k\leq x_k\}}\|_{L^2(\Omega)}\\
        & \leq C_\Omega \sup_{\xi\in \mathbb{R}^n} \|\nu(\cdot,\xi)\|_{C^{n}(\bar\Omega)} \|f\|_{L^2(\Omega)}.
    \end{align*}
    This completes the proof.
\end{proof}

\begin{lemma}\label{est_scal_prod}
    Let $\omega\subset \mathbb{R}^n$ be an open bounded set and $\mu$, $\nu\in C^\infty(\mathbb{R}\times \mathbb{S}^{n-1})$. Then, there exists a constant  $C_\Omega>0$, which depend only on $\Omega$, such that 
    \begin{equation*}
        |(R_\mu^* R_\nu D^\gamma f, f)_{L^2(\mathbb{R}^n)}| \leq C_\Omega \sup_{\omega\in \mathbb{S}^{n-1} }\|\mu(\cdot, \omega)\|_{C^{n+1}(\bar\Omega)}  \sup_{\omega\in \mathbb{S}^{n-1}} \left\|\nu(\cdot, \omega) \right\|_{C^{n -1 + |\gamma|}(\bar\Omega)}\|f\|_{L^2(\Omega)}.
    \end{equation*}
    for any $f\in C_0^\infty(\Omega)$ and multi-index any $\gamma$ with $0\leq|\gamma|\leq n-1$.
\end{lemma}

\begin{proof}
    Throughout this proof, $C_\Omega$ will serve as a universal positive constant depending only on $\Omega$, which may vary from line to line. Let $\chi\in C_0^\infty(\mathbb{R}^n)$ such that $\chi(x)=1$ for $x\in \Omega$. Then, for $f\in C_0^\infty(\Omega)$,
    \begin{equation*}
        |(R_\mu^* R_\nu D^\gamma f, f)_{L^2(\mathbb{R}^n)}| = |(\chi R_\mu^* R_\nu D^\gamma (\chi f), f)_{L^2(\mathbb{R}^n)}|\leq \|\chi R_\mu^* R_\nu D^\gamma (\chi f)\|_{L^2(\mathbb{R}^n)} \|f\|_{L^2(\mathbb{R}^n)}.
    \end{equation*}
    Let us investigate the first multiplier on the right-hand side. By Proposition 5.8.3 in \cite{SU-book}, $R^*_{\mu}R_{\nu}$ is a $\Psi \mathrm{DO}$ of order $1-n$ with the amplitude given by
    \begin{equation*}
        (2\pi - 1) \frac{\mu(x, \xi/|\xi|)\nu(y, \xi/|\xi|) + \mu(x, -\xi/|\xi|)\nu(y, -\xi/|\xi|)}{|\xi|^{n-1}}.
    \end{equation*}
    Then,
    \begin{equation*}
        \chi R_\mu^* R_\nu D^\gamma (\chi f) = Af + Bf
    \end{equation*}
    with
    \begin{equation*}
        Af(x) = C_\Omega\int_{\mathbb{R}^n} \int_{\mathbb{R}^n} e^{i(x - y)\xi} \chi(x) \frac{\mu(x, \xi/|\xi|)\nu(y, \xi/|\xi|)}{|\xi|^{n-1}} D_y^\gamma (\chi(y) f(y)) dy d\xi
    \end{equation*}
    and
    \begin{equation*}
        Bf(x) = C_\Omega\int_{\mathbb{R}^n} \int_{\mathbb{R}^n} e^{i(x - y)\xi} \chi(x) \frac{\mu(x, -\xi/|\xi|)\nu(y, -\xi/|\xi|)}{|\xi|^{n-1}} D_y^\gamma (\chi(y) f(y)) dy d\xi.
    \end{equation*}
    By integration by parts, we obtain
    \begin{equation*}
        |Af(x)| \leq C_\Omega \sum_{\alpha + \beta = \gamma} \left| \int_{\mathbb{R}^n} \int_{\mathbb{R}^n} e^{i(x - y)\xi} \chi(x) \xi^\alpha \frac{\mu(x, \xi/|\xi|) D_y^\beta \nu(y, \xi/|\xi|)}{|\xi|^{n-1}} \chi(y) f(y) dy d\xi \right|.
    \end{equation*}
    Let $\phi\in C^\infty(\mathbb{R}^n)$ be a function such that
    \begin{equation*}
        \phi(\xi) = 
        \begin{cases}
            0 & |\xi|\leq 1;\\
            1 & |\xi|\geq 2.
        \end{cases}
    \end{equation*}
    We denote
    \begin{equation*}
        A_{\alpha,\beta}^1f(x) =   \int_{\mathbb{R}^n} \int_{\mathbb{R}^n} e^{i(x - y)\xi} (1 - \phi(\xi))\chi(x) \xi^\alpha \frac{\mu(x, \xi/|\xi|) D_y^\beta \nu(y, \xi/|\xi|)}{|\xi|^{n-1}} \chi(y) f(y) dy d\xi 
    \end{equation*}
    and
    \begin{equation*}
        A_{\alpha,\beta}^2f(x) =   \int_{\mathbb{R}^n} \int_{\mathbb{R}^n} e^{i(x - y)\xi} \phi(\xi) \chi(x) \xi^\alpha \frac{\mu(x, \xi/|\xi|) D_y^\beta \nu(y, \xi/|\xi|)}{|\xi|^{n-1}} \chi(y) f(y) dy d\xi .
    \end{equation*}
    Then, the last inequality implies
    \begin{equation}\label{A}
        |Af(x)| \leq C_\Omega \sum_{\alpha + \beta = \gamma} \left(|A_{\alpha,\beta}^1f(x)| + |A_{\alpha,\beta}^2f(x)|\right).
    \end{equation}
    Let us estimate the $L^2$-norm of $A_{\alpha,\beta}^1f$. The kernel of $A_{\alpha,\beta}^1$ is given by 
    \begin{equation*}
        K_{\alpha,\beta}(x,y) = \int_{\mathbb{R}^n} e^{i(x - y)\xi} (1 - \phi(\xi)) \xi^\alpha \frac{\chi(x)\chi(y)\mu(x, \xi/|\xi|) D_y^\beta \nu(y, \xi/|\xi|)}{|\xi|^{n-1}}  d\xi.
    \end{equation*}
    We estimate
    \begin{multline*}
        \sup_{y\in \mathbb{R}^n} \int_{\mathbb{R}^n} |K_{\alpha,\beta}(x,y)| dx\\ \leq \|\chi\|_{L^\infty(\mathbb{R}^n)} \int_{\mathbb{R}^n} \xi^\alpha \frac{(1 - \phi(\xi))}{|\xi|^{n-1}} d\xi \sup_{\omega\in \mathbb{S}^{n-1}} \|\mu(\cdot,\omega)\|_{L^\infty(\mathbb{R}^n)} \sup_{\omega\in \mathbb{S}^{n-1}} \|\nu(\cdot,\omega)\|_{C^{|\beta|}(\mathbb{R}^n)}.
    \end{multline*}
    Hence, 
    \begin{equation*}
        \sup_{y\in \mathbb{R}^n} \int_{\mathbb{R}^n} |K_{\alpha,\beta}(x,y)| dx = C_\Omega \sup_{\omega\in \mathbb{S}^{n-1}} \|\mu(\cdot,\omega)\|_{L^\infty(\mathbb{R}^n)} \sup_{\omega\in \mathbb{S}^{n-1}} \|\nu(\cdot,\omega)\|_{C^{|\beta|}(\mathbb{R}^n)},
    \end{equation*}
    and similarly,
    \begin{equation*}
        \sup_{x\in \mathbb{R}^n} \int_{\mathbb{R}^n} |K_{\alpha,\beta}(x,y)| dy = C_\Omega \sup_{\omega\in \mathbb{S}^{n-1}} \|\mu(\cdot,\omega)\|_{L^\infty(\mathbb{R}^n)} \sup_{\omega\in \mathbb{S}^{n-1}} \|\nu(\cdot,\omega)\|_{C^{|\beta|}(\mathbb{R}^n)}.
    \end{equation*}
    Therefore, by Lemma 18.1.12 in \cite{Hormander3},
    \begin{equation}\label{A_ab_1}
        \|A_{\alpha,\beta}^1f\|_{L^2(\Omega)} \leq C_\Omega \sup_{\omega\in \mathbb{S}^{n-1}} \|\mu(\cdot,\omega)\|_{L^\infty(\mathbb{R}^n)} \sup_{\omega\in \mathbb{S}^{n-1}} \|\nu(\cdot,\omega)\|_{C^{|\beta|}(\mathbb{R}^n)} \|f\|_{L^2(\Omega)}.
    \end{equation}

    Next, we estimate the $L^2$-norm of $A_{\alpha,\beta}^2f$. We denote
    We denote
    \begin{equation*}
        \nu_{\alpha,\beta}(y,\xi) = \frac{\xi^\alpha D_y^\beta\nu(y, \xi/|\xi|)}{|\xi|^{|\alpha|}}\chi(y)
    \end{equation*}
    and
    \begin{equation*}
        v_{\alpha,\beta}(\xi) = \int_{\mathbb{R}^n} e^{-iy\xi} \nu_{\alpha,\beta}(y,\xi) f(y) dy,
    \end{equation*}
    so that
    \begin{multline*}
        A_{\alpha,\beta}^2f(x) = \int_{\mathbb{R}^n}  e^{ix\xi}  \phi(\xi)\frac{\chi(x)\mu(x, \xi/|\xi|)}{|\xi|^{n-1-|\alpha|}} v_{\alpha,\beta}(\xi) d\xi\\ = \int_{\mathbb{R}^n}  e^{ix\xi} \phi(\xi) \frac{\chi(x) \mu(x, \xi/|\xi|)}{|\xi|^{n-1-|\alpha|}} \mathcal{F}[\mathcal{F}^{-1}v_{\alpha,\beta}](\xi) d\xi.
    \end{multline*}
    Next, we note that 
    \begin{align*}
        \sum_{\tau \leq n + 1} \int_{\mathbb{R}^n} \left|D_x^\tau \left( \frac{\chi(x) \phi(\xi)\mu(x, \xi/|\xi|)}{|\xi|^{n-1-|\alpha|}} \right)\right|dx 
        & \leq C_\Omega \sup_{\omega\in \mathbb{S}^{n-1} }\|\chi(\cdot)\mu(\cdot, \omega)\|_{C^{n+1}(\mathbb{R}^n)}\\
        &\leq C_\Omega \sup_{\omega\in \mathbb{S}^{n-1} }\|\mu(\cdot, \omega)\|_{C^{n+1}(\bar\Omega)}.
    \end{align*}
    Therefore, Theorem 18.1.11' in \cite{Hormander3} implies that 
    \begin{align*}
        \int_{\mathbb{R}^n} \Big|  \int_{\mathbb{R}^n}  e^{ix\xi}  \frac{\chi(x) \phi(\xi)\mu(x, \xi/|\xi|)}{|\xi|^{n-1-|\alpha|}} & \mathcal{F}[\mathcal{F}^{-1}v_{\alpha,\beta}](\xi) d\xi \Big|^2 dx\\ 
        &\leq C_\Omega \sup_{\omega\in \mathbb{S}^{n-1} }\|\mu(\cdot, \omega)\|_{C^{n+1}(\bar\Omega)}^2 \|\mathcal{F}^{-1}v_{\alpha,\beta}\|_{L^2(\mathbb{R}^{n})}^2 \\
        &\leq C_\Omega \sup_{\omega\in \mathbb{S}^{n-1} }\|\mu(\cdot, \omega)\|_{C^{n+1}(\bar\Omega)}^2 \|v_{\alpha,\beta}\|_{L^2(\mathbb{R}^{n})}^2.
    \end{align*}
    Hence,
    \begin{equation*}
        \|A_{\alpha,\beta}^2f\|_{L^2(\mathbb{R}^{n})} \leq C_\Omega \sup_{\omega\in \mathbb{S}^{n-1} }\|\mu(\cdot, \omega)\|_{C^{n+1}(\bar\Omega)} \|v_{\alpha,\beta}\|_{L^2(\mathbb{R}^{n})}.
    \end{equation*}
    By Lemma \ref{L2_norm_wF_trans},
    \begin{equation}\label{v_alphabeta}
        \|v_{\alpha,\beta}\|_{L^2(\mathbb{R}^{n})} \leq C_\Omega \sup_{\xi\in \mathbb{R}^n} \|\nu_{\alpha,\beta}(\cdot,\xi)\|_{C^n(\bar\Omega)} \|f\|_{L^2(\Omega)}.
    \end{equation}
    We estimate
    \begin{align*}
        \sup_{\xi\in \mathbb{R}^n} \|\nu_{\alpha,\beta}(\cdot,\xi)\|_{C^n(\bar\Omega)}& = \sup_{\xi\in \mathbb{R}^n} \left\|\frac{\xi^\alpha D^\beta\nu(\cdot, \xi/|\xi|)}{|\xi|^{|\alpha|}}\chi(\cdot)\right\|_{C^n(\bar\Omega)} \leq C_\Omega \sup_{\xi\in \mathbb{R}^n} \left\|D^\beta\nu(\cdot, \xi/|\xi|) \right\|_{C^n(\bar\Omega)}\\
        &\leq C_\Omega\sup_{\xi\in \mathbb{S}^{n-1}} \left\|\nu(\cdot, \xi) \right\|_{C^{n + |\beta|}(\bar\Omega)}.
    \end{align*}
    Therefore,
    \begin{equation*}
        \|A_{\alpha,\beta}^2f\|_{L^2(\mathbb{R}^{n})} \leq C_\Omega \sup_{\omega\in \mathbb{S}^{n-1} }\|\mu(\cdot, \omega)\|_{C^{n+1}(\bar\Omega)}  \sup_{\xi\in \mathbb{S}^{n-1}} \left\|\nu(\cdot, \xi) \right\|_{C^{n + |\beta|}(\bar\Omega)} \|f\|_{L^2(\Omega)}.
    \end{equation*}
    Therefore, by \eqref{A} and \eqref{A_ab_1}, we obtain
    \begin{equation*}
        \|Af\|_{L^2(\mathbb{R}^n)} \leq C_\Omega \sup_{\omega\in \mathbb{S}^{n-1} }\|\mu(\cdot, \omega)\|_{C^{n+1}(\bar\Omega)}  \sup_{\omega\in \mathbb{S}^{n-1}} \left\|\nu(\cdot, \omega) \right\|_{C^{n -1 + |\gamma|}(\bar\Omega)} \|f\|_{L^2(\Omega)}.
    \end{equation*}
    Similarly, we estimate $\|Bf\|_{L^2(\mathbb{R}^n)}$ and conclude that
    \begin{equation*}
       \| \chi R_\mu^* R_\nu D^\alpha (\chi f)\|_{L^2(\mathbb{R}^n)} \leq C_\Omega \sup_{\omega\in \mathbb{S}^{n-1} }\|\mu(\cdot, \omega)\|_{C^{n+1}(\bar\Omega)}  \sup_{\omega\in \mathbb{S}^{n-1}} \left\|\nu(\cdot, \omega) \right\|_{C^{n -1 + |\gamma|}(\bar\Omega)}\|f\|_{L^2(\Omega)}.
    \end{equation*}
    This completes the proof.
\end{proof}

Now, we prove Theorem \ref{w_Radon_t_est}.

\begin{proof}[Proof of Theorem \ref{w_Radon_t_est}]
        Let $f\in C_0^\infty(\Omega)$ and $\Omega \subset \mathbb{R}^n$ compact. We denote $L=(n-1)/2$ and estimate
        \begin{equation*}
            \|R_\mu f\|_{L^2(\mathbb{S}_\omega^{n-1}; H^{L}(\mathbb{R}_\sigma))}^2 = \sum_{l=0}^{L} \|\partial_p^l R_\mu f\|_{L^2(\mathbb{R}\times \mathbb{S}^{n-1})} \leq \sum_{l=0}^{2L} |(R_\mu^*\partial_p^l R_\mu f,f)_{L^2(\mathbb{R}^n)}|.
        \end{equation*}
        We note that
        \begin{align*}
            \partial_p^l R_\mu f(p,\omega)& = \partial_p^l \left( \int_{\omega^\perp} \mu (p\omega + y, \omega) f(p\omega + y) dy \right)\\
            &= \sum_{|\alpha| + |\beta|=l} C_{\alpha,\beta} \int_{\omega^\perp} D^\alpha f(p\omega + y) D^\beta \mu (p\omega + y, \omega)  \omega^{\alpha+\beta} dy.
        \end{align*}
    Let us set
    \begin{equation*}
        \mu_{\alpha,\beta}(x,\omega) = \omega^{\alpha+\beta} D^\beta \mu(x,\omega)
    \end{equation*}
    so that, the previous equality gives
    \begin{equation*}
        \partial_p^l R_\mu = \sum_{|\alpha| + |\beta| = l} C_{\alpha,\beta} R_{\mu_{\alpha,\beta}} D^\alpha
        \quad \text{and} \quad
        R_\mu^*\partial_p^l R_\mu = \sum_{|\alpha| + |\beta| = l} C_{\alpha,\beta} R_\mu^* R_{\mu_{\alpha,\beta}} D^\alpha. 
    \end{equation*}
    By Lemma \ref{est_scal_prod}, we estimate
    \begin{align*}
        |(R_\mu^* R_{\mu_{\alpha,\beta}} D^\alpha f, f)_{L^2(\Omega)}|& \leq C_\Omega \sup_{\omega\in \mathbb{S}^{n-1}} \|\mu(\cdot,\omega)\|_{C^{n+1}(\bar\Omega)} \sup_{\omega\in \mathbb{S}^{n-1}} \|D^{\beta}\mu(\cdot,\omega)\|_{C^{n - 1 + |\alpha|}(\bar\Omega)} \|f\|_{L^2(\Omega)}^2\\
        &\leq C_\Omega \sup_{\omega\in \mathbb{S}^{n-1}} \|\mu(\cdot,\omega)\|_{C^{2n - 2}(\bar\Omega)}^2 \|f\|_{L^2(\Omega)}^2.
    \end{align*}
    This completes the proof.
\end{proof}

\section*{Acknowledgment}
The first author was partially supported by the Ministry of Education and Science of the Republic of Kazakhstan under grant AP22683207.  P.S. partly supported by  NSF  Grant DMS-2154489. L.O. and M.N. were supported by the European Research Council of the European Union, grant 101086697 (LoCal), and the Research Council of Finland, grants 347715 and 353096. Views and opinions expressed are those of the authors only and do not necessarily reflect those of the European Union or the other funding organizations. Neither the European Union nor the other funding organizations can be held responsible for them.

\bibliographystyle{abbrv}
\bibliography{Plamen-ref}

\end{document}